\documentclass[a4paper,10pt,reqno]{amsart}

\usepackage[UKenglish]{babel}
\usepackage[hmargin=2.5cm,bmargin=2.5cm,tmargin=2.2cm]{geometry}
\usepackage{amsmath}
\usepackage{amssymb}
\usepackage{amsthm}
\usepackage{verbatim}
\usepackage{stackrel}
\usepackage[arrow, matrix, curve]{xy} %for diagrams of implications

\usepackage[shortlabels]{enumitem}

\usepackage{comment,cite}	
\usepackage{hyperref}
\hypersetup{
    colorlinks=true,
    linkcolor=blue,
    citecolor=red,
    filecolor=magenta,      
    urlcolor=cyan,
    pdftitle={BIBO stability: From finite to infinite dimensions},
    bookmarks=true,
    linktocpage=true,
}

\usepackage[utf8]{inputenc}
\usepackage{booktabs,tabularx,capt-of}
\usepackage{mathtools}
\usepackage[thmtools-compat]{keytheorems}
\usepackage{bbm}
\usepackage{orcidlink}

\usepackage{marginnote}
\usepackage{color}

\usepackage{stmaryrd}
\usepackage{tikz}
\usepackage{tikz-cd}

\newcommand{\bbC}{\mathbb{C}}

\newcommand{\bbN}{\mathbb{N}}

\newcommand{\bbQ}{\mathbb{Q}}
\newcommand{\bbR}{\mathbb{R}}

\newcommand{\bbT}{\mathbb{T}}

\newcommand{\calF}{\mathcal{F}}

\newcommand{\calL}{\mathcal{L}}

\newcommand{\calS}{\mathcal{S}}

\DeclareMathOperator{\id}{id} % identity operator
\DeclareMathOperator{\one}{{\mathbbm{1}}} % constant function with value one
\DeclareMathOperator{\re}{Re} % real part
\newcommand{\argument}{\mathord{\,\cdot\,}} % argument dot for functions (with correct spacing)
\newcommand{\dx}{\;\mathrm{d}} % differential (for use at the end of integrals)
\newcommand{\dxShort}{\mathrm{d}} % differential without extra space (for use in differential quotients)
\newcommand{\norm}[1]{\left\lVert #1 \right\rVert} % norm
\newcommand{\modulus}[1]{\left\lvert #1 \right\rvert} % modulus
\newcommand{\duality}[2]{\left\langle#1\, ,\, #2\right\rangle} % duality / scalar product
\newcommand{\dom}[1]{\operatorname{dom}\left(#1\right)} % domain of an operator
\DeclareMathOperator{\Ima}{Rg} % image / range of a mapping
\newcommand\restrict[1]{\raisebox{-.5ex}{$|$}_{#1}} %restriction of a map

\newcommand{\resSet}{\rho}
\newcommand{\spr}{r} % spectral radius
\newcommand{\gbd}{\omega_0} % growth bound of a C_0-semigroup

\newcommand{\systemNode}[4]{\Sigma\left(#1,\allowbreak #2,\allowbreak #3,\allowbreak #4\right)}
\newcommand{\systemNodeSt}{\systemNode{A}{B}{C}{\mathbf{G}}}
\newcommand{\resolvent}[2]{\left( #1  - #2 \right)^{-1}}
\newcommand{\lpSpaceDT}[3]{L^{#1}(#2; #3)}
\newcommand{\lpSpaceD}[2]{L^{#1}(#2)}
\newcommand{\lpSpace}[1]{L^{#1}}

\newcommand{\loc}{\mathrm{loc}}

\newcommand{\Implies}[2]{``\ref{#1} $\Rightarrow$ \ref{#2}'':}

\theoremstyle{definition}
\newtheorem{definition}{Definition}[section]
\newtheorem{remark}[definition]{Remark}
\newtheorem{remarks}[definition]{Remarks}
\newtheorem*{remark*}{Remark}
\newtheorem*{remarks*}{Remarks}
\declaretheorem[style=definition,qed=$\blacktriangle$,sibling=definition]{example}

\theoremstyle{plain}
\newtheorem{proposition}[definition]{Proposition}
\newtheorem{lemma}[definition]{Lemma}
\newtheorem{theorem}[definition]{Theorem}
\newtheorem{corollary}[definition]{Corollary}

\numberwithin{equation}{section} % enumerate formulas within sections

\begin{document}

\title[BIBO stability: From finite to infinite dimensions]{BIBO Stability in Infinite Dimensions: Extensions and Limitations of Finite-Dimensional Theory}
\author[S.~Arora]{Sahiba Arora \orcidlink{0000-0003-1973-8358}}
\address[Sahiba Arora]{Institut für Analysis, 
Leibniz Universität Hannover, Welfengarten 1, 30167 Hannover, Germany}
\email{sahiba.arora@math.uni-hannover.de}
\author[A.A.~Wierzba]{Alexander A. Wierzba \orcidlink{0000-0002-2013-0246}}
\address[Alexander A. Wierzba]{Fakultät für Mathematik und Naturwissenschaften, Bergische Universität Wuppertal, Gaußstraße 20, 42119 Wuppertal, Germany}

\email{alexander.wierzba@uni-wuppertal.de}
%\thanks{This work was partially funded by the Deutsche Forschungsgemeinschaft (DFG, German Research Foundation) – Project-ID 531152215 – CRC 1701.}
\thanks{The second author acknowledges funding by the Deutsche Forschungsgemeinschaft (DFG, German Research Foundation) – Project-ID 531152215 – CRC 1701.}

\date{\today}
\begin{abstract}
The study of bounded-input bounded-output (BIBO) stability of
infinite-dimensional systems gives rise to difficulties that do not
occur in the classical finite-dimensional setting.
In this paper, we resolve several questions concerning the
definition and characterisation of BIBO stability in the framework
of system nodes.
In particular, we show that $\mathrm C^\infty$-BIBO stability
need not imply $\lpSpace{\infty}$-BIBO stability in general,
but that the two notions are equivalent when the output space
is reflexive.
Furthermore, we investigate to what extent classical results extend
to the infinite-dimensional setting, including the relation between
BIBO stability and exponential stability under controllability and
observability assumptions.
\end{abstract}

\maketitle

\section{Introduction}
    \label{sec:intro}

Bounded-input bounded-output (BIBO) stability is a classical notion in the study of finite-dimensional linear systems 
\begin{equation*}
    \begin{array}{ccc}
        \dot{x}(t)&=&Ax(t)+Bu(t)\\
        y(t)&=&Cx(t)+Du(t)
    \end{array}
    \quad \text{for } t\ge0,
\end{equation*}
where $A\in\bbC^{n_X \times n_X}$, $B\in\bbC^{n_X \times n_U}$, $C\in\bbC^{n_Y \times n_X}$, and $D\in\bbC^{n_Y \times n_U}$ are matrices of appropriate size. With zero initial state, such a system is called BIBO stable if
there exists a constant $c>0$ such that
\[
    \norm{y}_{\lpSpaceDT{\infty}{[0,t]}{\bbC^{n_Y}}}
    \le
    c\norm{u}_{\lpSpaceDT{\infty}{[0,t]}{\bbC^{n_U}}}
    \qquad\text{for every }t>0
\]
and every locally essentially bounded input $u$.
Thus, bounded inputs produce bounded outputs with a gain independent
of the time horizon.
BIBO stability has applications, for instance, in signal processing
\cite{b_TanJian19_DigitalSignalProcessing} and the design of feedback
controllers~\cite{b_Nise20_ControlSystemsEngineering}.

The corresponding concept for \emph{infinite-dimensional} systems, as they appear for instance in the study of boundary control systems modelled by partial differential equations, has similarly been considered~\cite{a_CallierDesoer1978, a_WangCobb96BIBOTimeInv, a_WangCobb03BIBOTimeVar} and has lately seen renewed interest~\cite{a_UnserNoteOnBiboStability, a_AbusaksakaPartingtonBIBO2013, SchwenningerWierzbaZwart2024, a_HastirEtAl23NonlinearBIBO}. One motivation is its role in control techniques such as funnel control
\cite{a_IlchmannRyanTrennFunnel2005,
a_BergerPucheSchwenningerFunnel2020}.
In particular,~\cite{SchwenningerWierzbaZwart2024} initiated a systematic
study of BIBO stability in the framework of system nodes, which includes
boundary control systems and allows unbounded control and observation
operators.
We continue this study, focusing on the distinctions that arise when
the input or output space is infinite-dimensional and on the relation
between BIBO stability and stability of the state dynamics.

\subsection{Contributions}

In finite dimensions, BIBO stability admits a classical
characterisation, and its relation to exponential stability is
well understood.
In infinite dimensions, both the choice of solution concept and the
properties of the input and output spaces become relevant.
We investigate which finite-dimensional results extend to this
setting, which fail, and which additional assumptions recover them.
Our contributions fall into four groups.
\begin{enumerate}[\upshape (1)]
    \item 
    {\bf Different notions of BIBO stability and characterisation via measures} 
    (Section~\ref{sec:BIBOnotionsAndImpulseResponse}).
    The input class and the solution concept give rise to different
    notions of BIBO stability.
    For system nodes with finite-dimensional input and output
    spaces, these notions are equivalent and hold precisely when
    the transfer function is the Laplace transform of an
    operator-valued measure of bounded total variation.
    We show the following:
    \begin{enumerate}
        \item 
        An $\lpSpace{\infty}$-BIBO stable system need not have a transfer
        function that is the Laplace transform of a measure of bounded
        total variation (Example~\ref{ex:systemNodeCounterexampleBddVariation}).
        
        \item 
        A $\mathrm C^\infty$-BIBO stable system need not be
        $\lpSpace{\infty}$-BIBO stable (Example~\ref{ex:c-infty-but-not-l-infty}),
        answering a question left open in
        \cite[Section~7]{a_SchwenningerWierzba24_ADualNotionToBIBOStability}.
        
        \item 
        For a $\mathrm C^\infty$-BIBO stable system,
        $\lpSpace{\infty}$-BIBO stability is equivalent to the
        existence of a locally Bochner-integrable representative
        for every zero-state generalised output generated by an
        input in $\lpSpace{\infty}_{\loc}(\bbR_+; U)$
        (Proposition~\ref{prop:cinfty-bibo-existing-linfty-output-estimate}
        and Corollary~\ref{cor:cinfty-bibo-plus-linfty-outputs-implies-linfty-bibo}).

        \item Consequently, when the output space is reflexive,
        $\mathrm C^\infty$-BIBO stability implies $\lpSpace{\infty}$-BIBO
        stability 
        (Theorem~\ref{thm:reflexive-output-cinfty-implies-linfty-bibo}),
        yielding corresponding $\lpSpace{\infty}$-BIBO criteria.
    \end{enumerate}
    
    \item 
    {\bf BIBO and exponential stability} 
    (Section~\ref{sec:BIBOandStability}). 
    In finite dimensions, exponential stability implies BIBO stability,
    and the converse holds for controllable and observable systems.
    The first implication is known to fail for general
    infinite-dimensional systems.
    We investigate the converse under different controllability
    and observability assumptions:
    \begin{enumerate}
        \item 
        BIBO stability of an approximately controllable and observable system does not imply exponential stability (Example~\ref{ex:simple-counterex-approx-obs-cont-exp-stable}).
        In fact, even an analytic semigroup can be unbounded when
        the system is $\lpSpace{\infty}$-BIBO stable, approximately controllable,
        and finite-time exactly observable
        (Example~\ref{ex:analytic-counterex-approx-obs-cont-not-strongly-stable}).

        \item 
        In contrast, $\lpSpace{\infty}$-BIBO stability does imply exponential
        stability under finite-time $\lpSpace{\infty}$-null-controllability
        and finite-time final observability, together with the
        corresponding control and observation admissibility assumptions
        (Proposition~\ref{prop:null-cont-final-obs-exponStab}).
    \end{enumerate}
    
    \item 
    {\bf BIBO stability of derived scalar-input scalar-output (SISO) systems} 
    (Section~\ref{sec:BIBOSISO}).
    Restricting the input to a fixed direction and applying a
    functional to the output produces a SISO system.
    These derived systems have finite-dimensional input and output
    spaces, so their BIBO stability can be characterised by measures
    of bounded total variation.
    However, uniform BIBO stability of all derived SISO systems is in general not sufficient for BIBO stability of the original system (Example~\ref{exa:siso-bibo-bounded-B-C}). This remains false even when the underlying semigroup is
    exponentially stable (Example~\ref{exa:siso-bibo-unbounded-B-C}).
    
    \item 
    {\bf BIBO stability from scalar-output systems}
    (Section~\ref{sec:BIBOfromScalarOutput}).
    Scalarising only the output yields a positive result:
    if every derived scalar-output system is
    $\mathrm C^\infty$-BIBO stable, then so is the original system
    (Theorem~\ref{thm:OutputScalarisationGivesBIBO}).
    We use this result to derive sufficient conditions in terms of
    integrability of scalarised kernels.
    Additional assumptions on kernel regularity or reflexivity of
    the output space yield corresponding $\lpSpace{\infty}$-BIBO criteria.
\end{enumerate}

\subsection{Notation}

We write $\bbR_+=[0,\infty)$ and $\bbC_\alpha \coloneqq \left\lbrace z \in \bbC \;\middle|\; \re{z} > \alpha \right\rbrace$ for $\alpha\in \bbR$ with the convention $\bbC_{-\infty}\coloneqq \bbC$. All Banach spaces are considered over the field of complex numbers.
The dual of a Banach space $X$ is denoted by $X'$,
with duality pairing $\duality{x'}{x}=x'(x)$.
Let $\calL(X,Y)$ denote the space of bounded linear operators from a Banach space $X$ to a Banach space $Y$. We write $\calL(X)=\calL(X,X)$.
For $T\in\calL(X,Y)$, its Banach-space adjoint is denoted by
$T'\in\calL(Y',X')$.
The domain of an unbounded operator $A$ is denoted by $\dom{A}$.

All $\lpSpace{p}$-spaces of Banach-space-valued functions are Bochner
spaces with respect to Lebesgue measure unless another measure
is specified.
In particular, $\lpSpace{\infty}$ consists of equivalence classes of
strongly measurable, essentially bounded functions.
The notation $\lpSpace{p}_{\loc}(\bbR_+; E)$ means that the restriction
to $[0,t]$ belongs to $\lpSpaceDT{p}{[0,t]}{E}$ for every $t>0$.
We denote by $\one_M$ the indicator function of a measurable set $M$.

If $A$ is the generator of a strongly continuous semigroup $\bbT=(\bbT(t))_{t\ge 0}$ on a Banach space $X$, let $\gbd(\bbT)=\inf\{\omega\in\bbR\colon\sup_{t>0}e^{-t\omega}\norm{\bbT(t)}<\infty\}$ denote its growth bound. 
The semigroup is called exponentially stable if $\gbd(\bbT)<0$. 
The space $X_1$ is defined as $\dom{A}$ 
equipped with the graph norm and the extrapolation space of the semigroup $X_{-1}$ is defined to be the completion of $X$ with respect to the norm $\norm{x}_{X_{-1}} \coloneqq \norm{\resolvent{\beta}{A} x}_X$, for a fixed number $\beta$ in the resolvent set $\resSet(A)$. 
Different choices of $\beta$ give rise to equivalent norms. The canonical embeddings
$X_1\hookrightarrow X\hookrightarrow X_{-1}$
are continuous and dense.
We denote by $(\bbT_{-1}(t))_{t\ge0}$ the extension of
$(\bbT(t))_{t\ge0}$ to $X_{-1}$ and by $A_{-1}$ its generator,
with $\dom{A_{-1}}=X$.
For more details on these concepts see~\cite[Chapter~II]{EngelNagel2000}.

\subsection{System nodes and BIBO stability}

Let $U$, $X$, and $Y$ be Banach spaces. A \emph{system node} on $(U,X,Y)$ is a collection $\systemNodeSt$ where $A: \dom{A} \subseteq X \rightarrow X$ is the generator of a strongly continuous semigroup $\bbT$ on $X$, $B \in \calL(U, X_{-1})$, $C \in \calL(X_1, Y)$, and $\mathbf{G}: \bbC_{\gbd(\bbT)} \rightarrow \calL(U,Y)$ is an analytic function that satisfies 
\begin{equation}
    \label{eq:transferFunctionDifferenceRelation}
    \mathbf{G}(\alpha) - \mathbf{G}(\beta) = C\left[ \resolvent{\alpha}{A_{-1}} - \resolvent{\beta}{A_{-1}} \right] B
    \quad
    \text{for all }\alpha, \beta \in \bbC_{\gbd(\bbT)}
\end{equation}
The spaces $U$, $X$, and $Y$ are usually referred to as \emph{input space}, \emph{state space}, and \emph{output space}, respectively. Furthermore, $B$ is called the \emph{input operator}, $C$ the \emph{output operator}, and $\mathbf{G}$ the \emph{transfer function} of the system node. 

Throughout, $\systemNodeSt$ denotes a system node on
Banach spaces $(U,X,Y)$ as above.
We write $\bbT$ for the $C_0$-semigroup generated by $A$,
and $X_1$ and $X_{-1}$ for the associated domain and
extrapolation spaces, respectively.

Formally, a system node represents the differential equation 
\begin{equation}
\label{eq:systemNodeDifferentialEqns}
    \begin{split}
        \Dot{x}(t) &= A_{-1} x(t) + B u(t) \\
        y(t) &= C\&D \begin{bmatrix} x(t) \\ u(t) \end{bmatrix},
    \end{split}
\end{equation}
where $C \& D: \dom{C \& D} \subseteq X \times U \rightarrow Y$ is the \emph{combined observation/feedthrough operator} defined as
\begin{align*}
    \dom{C\&D} &= \left\{  \begin{bmatrix} x \\ u \end{bmatrix} \in X \times U \colon A_{-1}x + Bu \in X \right\}\\
    \begin{bmatrix} x \\ u \end{bmatrix} &\mapsto C\left[x-\resolvent{\beta}{A_{-1}} B u \right] + \mathbf{G}(\beta) u
\end{align*}
for any choice of $\beta \in \bbC_{\gbd(\bbT)}$.
We equip $\dom{C\&D}$ with the norm
\[
    \norm{
        \begin{bmatrix}x\\u\end{bmatrix}
    }_{\dom{C\&D}}
    =
    \norm{x}_X+\norm{u}_U+\norm{A_{-1}x+Bu}_X.
\]
With this norm, $\dom{C\&D}$ is a Banach space and
$C\&D:\dom{C\&D}\to Y$ is bounded.
For every $\begin{bmatrix}x\\u\end{bmatrix}\in\dom{C\&D}$,
the vector $x-\resolvent{\beta}{A_{-1}}Bu$ belongs to $X_1$,
and the transfer-function identity~\eqref{eq:transferFunctionDifferenceRelation}
ensures that the definition of $C\&D$ is independent of the choice of
$\beta\in\bbC_{\gbd(\bbT)}$.

Equation~\eqref{eq:systemNodeDifferentialEqns} gives rise to two different types of solutions. First, \emph{classical solutions}, that is, triples $(u,x,y) \in \mathrm C(\bbR_+; U) \times \mathrm C^1(\bbR_+; X) \times \mathrm C(\bbR_+; Y)$ such that $\left[ \begin{smallmatrix}
    x(t) \\ u(t)
\end{smallmatrix} \right] \in \dom{C\&D}$ and~\eqref{eq:systemNodeDifferentialEqns} is satisfied for all $t \ge 0$. Second, we consider \emph{generalised distributional solutions},
defined as follows. For $u\in \lpSpace{1}_{\loc}(\bbR_+; U)$ and $x_0\in X$, define
\[
    x(t)
    =
    \bbT(t)x_0
    +
    \int_0^t\bbT_{-1}(t-s)Bu(s)\dx s.
\]
Then $x\in \mathrm C(\bbR_+;X_{-1})$.
By~\cite[Lemma~4.7.9]{Staffans2005},
\[
    t\mapsto
    \int_0^t(t-s)
    \begin{bmatrix}x(s)\\u(s)\end{bmatrix}\dx s
\]
belongs to $\mathrm C(\bbR_+;\dom{C\&D})$.
Hence
\[
    F(t)
    \coloneqq
    (C\&D)\int_0^t(t-s)
    \begin{bmatrix}x(s)\\u(s)\end{bmatrix}\dx s
\]
is a continuous $Y$-valued function.
The generalised output is the distribution $y=F''$ on
$(0,\infty)$.
Following~\cite[Section~2.2]{SchwenningerWierzbaZwart2024},
we use its action on $Y'$-valued test functions:
\[
    y(\varphi)
    =
    \int_0^\infty
        \duality{\varphi''(t)}{F(t)}
    \dx t,
    \qquad
    \varphi\in\mathrm C_c^\infty((0,\infty);Y').
\]
For a scalar test function
$\psi\in\mathrm C_c^\infty((0,\infty))$, we also write
$y(\psi)\coloneqq\int_0^\infty\psi''(t)F(t)\dx t\in Y$.
These actions are compatible:
$y(\psi y')=\duality{y'}{y(\psi)}$
for every such $\psi$ and every $y'\in Y'$.

A function $g\in\lpSpace{1}_{\loc}(\bbR_+;Y)$ represents $y$
if
$y(\varphi)=\int_0^\infty\duality{\varphi(t)}{g(t)}\dx t$
for every $\varphi\in\mathrm C_c^\infty((0,\infty);Y')$.
Equivalently,
$y(\psi)=\int_0^\infty\psi(t)g(t)\dx t$
for every scalar test function
$\psi\in\mathrm C_c^\infty((0,\infty))$.

For each input and initial state, the generalised solution is unique,
and every classical solution is also a generalised solution~\cite[Lemma~4.7.8 and Theorem~4.7.11]{Staffans2005}.
The zero-state input-output map is causal: inputs agreeing almost
everywhere on $(0,t)$ have outputs agreeing distributionally there.

Throughout, inputs in $\mathrm C_c^\infty((0,\infty);U)$ are extended
by zero to $\bbR_+$.
Since these inputs vanish near $t=0$, each admits a unique
classical solution with initial state $x(0)=0$;
see~\cite[Lemma~4.7.8]{Staffans2005}.
The solution concepts above lead to the following notions of
BIBO stability
\cite[Definitions~3.1 and~3.2]{SchwenningerWierzbaZwart2024}.

\begin{definition}
    \label{def:BIBOStability}
    Let $\systemNodeSt$ be a system node. 
    \begin{enumerate}[\upshape (a)]
        \item 
        The system node is called $\mathrm C^\infty$-\emph{BIBO stable} if there exists $c > 0$ such that for each classical solution $(u,x,y)$ with $u \in \mathrm C_c^\infty\left( (0,\infty); U \right)$ and $x(0) = 0$, one has
        \begin{equation}
            \label{eq:BIBOInequality}
            \norm{y}_{\lpSpaceDT{\infty}{[0,t]}{Y}} \le c \norm{u}_{\lpSpaceDT{\infty}{[0,t]}{U}} 
            \quad \text{for all } t > 0.
        \end{equation}

        \item 
        The system node is called $\lpSpace{\infty}$-\emph{BIBO stable} if there exists $c > 0$ such that for each generalised distributional solution $(u,x,y)$ with $u \in \lpSpace{\infty}_{\loc}(\bbR_+; U)$ and $x(0) = 0$ the output distribution $y$ is represented by an element of
        $\lpSpace{\infty}_{\loc}(\bbR_+; Y)$, again denoted by $y$, and
        satisfies~\eqref{eq:BIBOInequality}.
    \end{enumerate}
\end{definition}

\begin{remarks}
\label{rem:BIBO-constants-estimates}
    Throughout these remarks, outputs correspond to zero initial state.
    \begin{enumerate}[\upshape (a)]
        \item For $\mathrm C^\infty$-BIBO stability it is equivalent to require the global estimate
        \begin{equation}
            \label{eq:globalBIBOInequality}
            \norm{y}_{\lpSpaceDT{\infty}{\bbR_+}{Y}} \le c \norm{u}_{\lpSpaceDT{\infty}{\bbR_+}{U}}
        \end{equation}
        for every $u \in \mathrm C_c^\infty\left( (0,\infty);U\right)$.
        Indeed, the local estimate implies the global one since $c$ is independent of $t$. Conversely, fix $u\in \mathrm C_c^\infty((0,\infty);U)$ and $t>0$.
        For  $\varepsilon>0$, choose $\chi_\varepsilon\in \mathrm C^\infty(\bbR_+;[0,1])$
        equal to one on $[0,t]$ and zero on $[t+\varepsilon,\infty)$.
        By causality and the global estimate~\eqref{eq:globalBIBOInequality},
        \[
            \norm{y_u}_{\lpSpaceDT{\infty}{[0,t]}{Y}}
            =
            \norm{y_{\chi_\varepsilon u}}_{\lpSpaceDT{\infty}{[0,t]}{Y}}
            \le
            c\norm{\chi_\varepsilon u}_{\lpSpaceDT{\infty}{\bbR_+}{U}}.
        \]
        Continuity of $u$ ensures
        $
            \norm{\chi_\varepsilon u}_{\lpSpaceDT{\infty}{\bbR_+}{U}}
            \longrightarrow
            \norm{u}_{\lpSpaceDT{\infty}{[0,t]}{U}}
        $
        as $\varepsilon\downarrow0$.
        Thus the local estimate~\eqref{eq:BIBOInequality} holds with the same constant.
        
        \item For $\lpSpace{\infty}$-BIBO stability, it is analogously equivalent to require that for any $u \in \lpSpaceDT{\infty}{\bbR_+}{U}$, the generalised output is represented by an element of $\lpSpaceDT{\infty}{\bbR_+}{Y}$ and satisfies~\eqref{eq:globalBIBOInequality}.

        Again the local estimate implies the global one since $c$ is independent of $t$. For the converse, fix $u\in \lpSpace{\infty}_{\loc}(\bbR_+; U)$ and $t>0$.
        Set $u_t \coloneqq u \one_{[0,t]}\in \lpSpaceDT{\infty}{\bbR_+}{U}$, so its output
        has a representative $y_{u_t}\in \lpSpaceDT{\infty}{\bbR_+}{Y}$.
        By causality, this representative also represents
        $y_u$ on $(0,t)$, and
        \[
            \norm{y_u}_{\lpSpaceDT{\infty}{[0,t]}{Y}}
            \le
            \norm{y_{u_t}}_{\lpSpaceDT{\infty}{\bbR_+}{Y}}
            \le
            c\norm{u_t}_{\lpSpaceDT{\infty}{\bbR_+}{U}}
            =
            c\norm{u}_{\lpSpaceDT{\infty}{[0,t]}{U}}.
        \]
        The representatives agree almost everywhere on
        overlapping intervals.
        Since $t>0$ was arbitrary, they define a representative
        of $y_u$ in $\lpSpace{\infty}_{\loc}(\bbR_+; Y)$.
        Hence,~\eqref{eq:BIBOInequality} holds.

         \item  For either notion of BIBO stability, the
        \emph{optimal BIBO constant} is the smallest $c\ge0$
        for which~\eqref{eq:BIBOInequality} holds for all
        inputs covered by that notion.
        Equivalently, it is the infimum of all admissible
        positive constants; this infimum also satisfies
        the estimate.
        By~(a) and~(b), the local and global estimates have the same optimal constants.
    \end{enumerate}
\end{remarks}

\section{Notions of BIBO stability and measures of bounded total variation}
\label{sec:BIBOnotionsAndImpulseResponse}

We begin with the finite-dimensional characterisation from
\cite[Theorems~3.4 and~3.5]{SchwenningerWierzbaZwart2024}.
Following~\cite[Remark~3.7]{SchwenningerWierzbaZwart2024},
we identify a Laplace transform with its analytic continuation
to the domain of the transfer function whenever such a
continuation exists.

\begin{theorem}
    \label{thm:finiteDimBIBOCharacterisation}
    Let $\systemNodeSt$ be a system node with finite dimensional input and output spaces $U$ and $Y$. Then the following are equivalent:
    \begin{enumerate}[\upshape(i)]
        \item\label{thm:finiteDimBIBOCharacterisation:itm:CInfty}
        $\systemNodeSt$ is $\mathrm C^\infty$-BIBO stable.
        
        \item\label{thm:finiteDimBIBOCharacterisation:itm:LInfty} 
        $\systemNodeSt$ is $\lpSpace{\infty}$-BIBO stable.

        \item\label{thm:finiteDimBIBOCharacterisation:itm:transfer}
        There exists a $\calL(U,Y)$-valued Borel measure $h$
        on $\bbR_+$ of bounded total variation such that
        $\mathbf{G}=\calL\{h\}$ on $\bbC_{\gbd(\bbT)}$.
    \end{enumerate}
\end{theorem}

The proof of this theorem in~\cite{SchwenningerWierzbaZwart2024}
uses convolution representations of bounded, causal, and
shift-invariant operators between scalar-valued function spaces.
Such representations do not extend in the same form to arbitrary
Banach-valued function spaces, so one cannot expect the full
equivalence to persist in infinite dimensions.
Nevertheless, the implication
\ref{thm:finiteDimBIBOCharacterisation:itm:transfer}$\Rightarrow$\ref{thm:finiteDimBIBOCharacterisation:itm:CInfty}
remains valid for arbitrary Banach input and output spaces,
with BIBO constant $\modulus{h}(\bbR_+)$.
The following examples show that neither
\ref{thm:finiteDimBIBOCharacterisation:itm:CInfty}$\Rightarrow$\ref{thm:finiteDimBIBOCharacterisation:itm:LInfty}
nor
\ref{thm:finiteDimBIBOCharacterisation:itm:LInfty}$\Rightarrow$\ref{thm:finiteDimBIBOCharacterisation:itm:transfer}
remains valid in general.

\begin{example}[$\lpSpace{\infty}$-BIBO stability without a measure of bounded total variation]
\label{ex:systemNodeCounterexampleBddVariation}
    Consider the system node $\systemNodeSt$ with spaces $X = U = Y = \mathrm C[0,1]$, operators
    \[
        A : X \rightarrow X, \quad f \mapsto\left( \xi \rightarrow - \xi \, f(\xi) \right), 
    \]
    $B = -A$ and $C = \id$, and transfer function $\mathbf{G}(s) = \resolvent{s}{A} B$.

    Since the operators are bounded and the feedthrough term is zero,
    every zero-state generalised output corresponding to
    $u\in \lpSpace{\infty}_{\loc}(\bbR_+; U)$ is represented by the
    continuous $Y$-valued function
    $
        y(t)(\xi)
        =
        \int_0^t \xi e^{-(t-s)\xi}u(s)(\xi)\dx s
    $,
    for $t\ge0, \xi\in[0,1]$.
    Consequently,
    \[
        \modulus{y(\tau)(\xi)} \le \norm{u}_{\lpSpaceDT{\infty}{[0,t]}{U}} \int_0^{\tau} \xi e^{-(\tau-s)\xi}\dx s = \norm{u}_{\lpSpaceDT{\infty}{[0,t]}{U}} (1-e^{-\tau \xi})
    \]
    for all $0 \le \tau \le t$ and for $\xi\in [0,1]$.
    In particular, $y \in \lpSpace{\infty}_{\loc}(\bbR_+; Y)$ and $\norm{y}_{\lpSpaceDT{\infty}{[0,t]}{Y}} \le \norm{u}_{\lpSpaceDT{\infty}{[0,t]}{U}}$ for all $t>0$. Thus, the system node is $\lpSpace{\infty}$-BIBO stable.

    Furthermore, set
    $
        k(t)\coloneqq C\bbT(t)B
    $
    for $t\ge0$. Since $A$ is bounded, $k$ is norm-continuous and is the inverse Laplace transform of $\mathbf{G}$ on $\bbC_0$.
    However,
    \[
        \norm{k(t)}_{\calL(U,Y)}
        =
        \sup_{\xi\in[0,1]}\xi e^{-t\xi}
        =
        \begin{cases}
            e^{-t}, & 0\le t\le1,\\[0.2em]
            1/(et), & t>1
        \end{cases}
    \]
    implies that
    $
        \int_0^\infty\norm{k(t)}_{\calL(U,Y)}\dx t=\infty
    $.
    If $\mathbf{G}=\calL\{h\}$ for a measure $h$ of bounded
    total variation, scalarisation and uniqueness of scalar
    Laplace transforms, after exponential weighting, give
    $h(E)=\int_E k(t)\dx t$ for every bounded Borel set
    $E\subseteq\bbR_+$.
    Hence $\modulus{h}([0,R])=\int_0^R\norm{k(t)}\dx t$ for every
    $R>0$, contradicting bounded total variation. \qedhere
\end{example}

\begin{remark}
    In Example~\ref{ex:systemNodeCounterexampleBddVariation},
    the following additional properties hold. 
    The terminology concerning controllability and observability
    is recalled at the beginning of Section~\ref{sec:BIBOandStability}.
    \begin{enumerate}[\upshape (a)]
        \item 
        The semigroup is not exponentially stable as $\norm{\bbT(t)} = 1$ for all $t \ge 0$.
        
        \item
        Since $C=\id$ and
        $\norm{\bbT(t)f}_X\ge e^{-t}\norm{f}_X$,
        the system is exactly $\lpSpace{p}$-observable in every time
        $\tau>0$, for every $p\in[1,\infty]$.

        \item 
        The impulse response has the norm-continuous density
        $k=C\bbT(\argument)B\in
        \lpSpace{1}_{\loc}(\bbR_+; \calL(U,Y))$,
        but $k\notin\lpSpaceDT{1}{\bbR_+}{\calL(U,Y)}$.
        
        \item 
        Choosing instead $X=U=Y= \{f\in \mathrm C[0,1]: f(0)=0\}$
        yields an $\lpSpace{\infty}$-BIBO stable system whose semigroup
        is strongly stable and which is approximately controllable
        and finite-time exactly observable.
        Indeed, $\bbT(t)f\to0$ in $X$ for every $f\in X
        $,
        and $B=-A$ has dense range because
        $\xi(\xi+\varepsilon)^{-1}f(\xi)\to f(\xi)$ uniformly
        as $\varepsilon\downarrow0$.
    \end{enumerate}
\end{remark}

\begin{example}[$\mathrm C^\infty$-BIBO stability without $\lpSpace{\infty}$-BIBO stability]
    \label{ex:c-infty-but-not-l-infty}
    Consider the transport system
    \begin{align*}
        \partial_t x(t,\theta)=\partial_\theta x(t,\theta), \qquad & t\ge 0, \theta\in(-\tfrac12,0)\\
        x(t,0)=u(t), \quad x(0,\theta)=0\qquad & t\ge0, \theta\in[-\tfrac12,0]\\
        y(t,r) = x(t,-r) - u(t) \qquad & t\ge 0, r \in [0,\tfrac12].
    \end{align*}
    Let $X= \lpSpaceD{1}{\left[-\frac12,0\right]}$, $U=\bbC$, and $Y= \mathrm C\left[0,\frac12\right]$. On $(U,X,Y)$ the above boundary control system can be realized as a system node $\systemNodeSt$ as follows: The operator $A$ is the generator of the nilpotent left-shift semigroup on $X$ and the control operator is given by $B= \delta_0 \in X_{-1}$.
    Moreover, the observation $C\in \calL(X_1,Y)$ is the reflection operator defined by $(Cf)(r)=f(-r)$ for all $r\in \left[0,\frac12\right]$. 
    The output equation specifies the combined
    observation/feedthrough operator by
    \[
        \dom{C\&D}
        =
        \left\{
            \begin{bmatrix}f\\v\end{bmatrix}
            \in W^{1,1}\bigl((-\tfrac12,0)\bigr)\times\bbC:
            f(0)=v
        \right\},
        \qquad
        \left(
            C\&D\begin{bmatrix}f\\v\end{bmatrix}
        \right)(r)
        =
        f(-r)-v.
    \]

    Let $u\in \mathrm C_c^\infty((0,\infty))$ and let $\widetilde u$
    denote its extension by zero to $\bbR$.
    The corresponding zero-state classical solution satisfies
    \[
        x(t,\theta)=\widetilde u(t+\theta),
        \qquad
        y(t)(r)=\widetilde u(t-r)-u(t).
    \]
    Hence,
    $
        \norm{ y }_{ \lpSpaceDT{\infty}{[0,t]}{Y}   }\le 2 \norm{u}_{\lpSpaceDT{\infty}{[0,t]}{U}}
    $
    for every $t>0$.
    Thus the system is $\mathrm C^\infty$-BIBO stable.

    Now, consider the bounded input 
    $u = \one_{[1,\infty)} \in \lpSpace{\infty}_{\loc}(\bbR_+; U)$ and
    let $(u,x,y)$ denote the corresponding zero-state generalised
    distributional solution.
    Extend $u$ by zero to negative times and define
    \[
        F(t)
        \coloneqq
        \frac12(t-1)^2\one_{[1,\infty)}(t),
        \qquad t\in\bbR.
    \]
    For $t\ge0$, this equals
    $\int_0^t(t-s)u(s)\dx s$.
    The distributional-output formula gives
    \[
        y(t)
        =
        \frac{\dxShort^2}{\dxShort t^2}
        \left(r\mapsto F(t-r)-F(t)\right).
    \]
    Consequently, for each fixed $r\in[0,\tfrac12]$,
    evaluation at $r$ gives the scalar distribution
    $u(\argument-r)-u(\argument)$.
    
    Suppose that $y$ were represented by an element of
    $\lpSpace{1}_{\loc}(\bbR_+; Y)$.
    Then, for each $q\in[0,\tfrac12]$, one has
    \[
        y(t)(q)= u(t-q)-u(t) \quad \text{for almost every } t\in \left(1,\frac32\right).
    \]
    In particular, there exists a set of full measure $M\subseteq \left(1,\frac32\right)$ such that
    \[
        y(t)(q)= u(t-q)-u(t) \quad \text{for all } t\in M, q\in \left[0,\frac12\right] \cap \bbQ.
    \]
    For fixed $t\in M$, this gives
    \[
        y(t)(q) 
        = 
        \begin{cases}
            0, \qquad &q< t-1, q\in \bbQ\\
           -1, \qquad &q> t-1, q\in \bbQ,
        \end{cases}
    \]
    which contradicts that $y(t)\in Y$.
\end{example}

Example~\ref{ex:c-infty-but-not-l-infty} shows that
$\mathrm C^\infty$-BIBO stability does not ensure the
existence of locally essentially bounded output representatives.
This raises two questions for a $\mathrm C^\infty$-BIBO
stable system node, with outputs corresponding to zero initial state:
\begin{enumerate}[\upshape (a)]
    \item
    If every $u\in \lpSpace{\infty}_{\loc}(\bbR_+; U)$ has an output
    representative in $\lpSpace{\infty}_{\loc}(\bbR_+; Y)$,
    must the system be $\lpSpace{\infty}$-BIBO stable?

    \item
    Whenever an individual $u\in \lpSpace{\infty}_{\loc}(\bbR_+; U)$
    has such an output representative, must it satisfy
    the BIBO estimate with the same constant as for smooth inputs?
\end{enumerate}
The following proposition answers both questions affirmatively.
In fact, local Bochner integrability of the output representative
already suffices to obtain essential boundedness and the same
BIBO estimate.

\begin{proposition}
    \label{prop:cinfty-bibo-existing-linfty-output-estimate}
    Let $\systemNodeSt$ be a system node which is
    $\mathrm C^\infty$-BIBO stable with constant $c\ge 0$.
    Let $t>0$ and let $u\in \lpSpace{1}_{\loc}(\bbR_+; U)$ satisfy
    $
        u\restrict{(0,t)}\in \lpSpaceDT{\infty}{[0,t]}{U}
    $.
    Denote by $y_u$ the zero-initial-state generalised output in the
    distributional sense. If $y_u$ is represented on $(0,t)$ by an element of
    $\lpSpaceDT{1}{[0,t]}{Y}$, then in fact $y_u\in \lpSpaceDT{\infty}{[0,t]}{Y}$ and
    \[
        \norm{y_u}_{\lpSpaceDT{\infty}{[0,t]}{Y}}
        \le
        c\norm{u}_{\lpSpaceDT{\infty}{[0,t]}{U}} .
    \]
\end{proposition}
The proof of this result is based upon the following continuity result.
\begin{lemma}
    \label{lem:L1-to-distributional-output-continuity}
    Let $\systemNodeSt$ be a system node and let
    $t>0$. 
    Extend every $u\in \lpSpaceDT{1}{[0,t]}{U}$ by zero to $\bbR_+$ and let
    $(u,x_u,y_u)$ be the generalised solution in the distributional sense with
    $x_u(0)=0$ for each such $u$. 
    Then there exists a constant $M_t>0$ such that
    \[
        \modulus{y_u(\varphi)}
        \le
        M_t
        \norm{u}_{\lpSpaceDT{1}{[0,t]}{U}}
        \norm{\varphi''}_{\lpSpaceDT{1}{[0,t]}{Y'}}
    \]
    for all $u\in \lpSpaceDT{1}{[0,t]}{U}$ and all
    $\varphi\in\mathrm C_c^\infty((0,t);Y')$.
\end{lemma}

\begin{proof}
    First of all,~\cite[Lemma~4.7.9]{Staffans2005}, applied with initial time $0$ and
    initial state $0$, gives
    \[
        \begin{bmatrix}
            K_u\\ L_u
        \end{bmatrix}
        \coloneqq 
        \int_0^{\argument} (\argument-s)
        \begin{bmatrix}
            x_u(s)\\ u(s)
        \end{bmatrix}
        \dx s
        \in
        \mathrm C([0,t];\dom{C\&D}).
    \]
    By the mild-solution formula, $u\mapsto x_u$ is
    bounded from $\lpSpaceDT{1}{[0,t]}{U}$ into
    $\mathrm C([0,t];X_{-1})$.
    Hence the map
    \[
        \lpSpaceDT{1}{[0,t]}{U}\ni u
        \mapsto
        \begin{bmatrix}K_u\\L_u\end{bmatrix}
        \in \mathrm C([0,t];X_{-1}\times U)
    \]
    is bounded.
    Its range lies in the Banach space
    $\mathrm C([0,t];\dom{C\&D})$, which embeds continuously
    into $\mathrm C([0,t];X_{-1}\times U)$.
    The closed graph theorem therefore yields $c_t'>0$ such that
    \[
        \sup_{0\le \tau\le t}
        \norm{
            \begin{bmatrix}K_u(\tau)\\L_u(\tau)\end{bmatrix}
        }_{\dom{C\&D}}
        \le
        c_t'\norm{u}_{\lpSpaceDT{1}{[0,t]}{U}}.
    \]

    Since
    $
        C\&D:\dom{C\&D}\to Y
    $
    is bounded,
    $
        F_u
        \coloneqq 
        (C\&D)
        \begin{bmatrix}
            K_u\\L_u
        \end{bmatrix}
        \in 
        \mathrm C([0,t];Y)
    $
    and there exists $c_t''>0$ such that
    $
        \norm{F_u}_{\mathrm C([0,t];Y)}
        \le
        c_t''
        \norm{u}_{\lpSpaceDT{1}{[0,t]}{U}}.
    $
    Since the output distribution is given by
    $
        y_u
        =
        \frac{\dxShort^2}{\dxShort \tau^2}F_u,
    $
    we obtain for $\varphi\in\mathrm C_c^\infty((0,t);Y')$,
    \begin{align*}
        \modulus{y_u(\varphi)}
        &=
        \modulus{
            \int_0^t
            \duality{\varphi''(\tau)}{F_u(\tau)}
            \dx \tau
        }\\                                                   
        &\le
        \norm{F_u}_{\mathrm C([0,t];Y)}
        \norm{\varphi''}_{\lpSpaceDT{1}{[0,t]}{Y'}}                            
        \le
        c_t''
        \norm{u}_{\lpSpaceDT{1}{[0,t]}{U}}
        \norm{\varphi''}_{\lpSpaceDT{1}{[0,t]}{Y'}}.
    \end{align*}
    This proves the asserted estimate with $M_t=c_t''$.
\end{proof}

\begin{proof}[Proof of Proposition~\ref{prop:cinfty-bibo-existing-linfty-output-estimate}]
    Put
    $
        M\coloneqq\norm{u}_{\lpSpaceDT{\infty}{[0,t]}{U}}
    $.
    By replacing $u$ by $\one_{[0,t]}u$, we do not change the generalised output
    on $(0,t)$. Thus we may assume for the following
    approximation argument that $u$ is supported in $[0,t]$.
    Choose a sequence
    $  u_n\in \mathrm C_c^\infty((0,\infty);U)$
    such that
    \[
        u_n\to u
        \quad\text{in }\lpSpaceDT{1}{[0,t]}{U},
        \qquad
        \norm{u_n}_{\lpSpaceDT{\infty}{[0,t]}{U}}\le M.
    \]
    
    Let $(u_n,x_n,y_n)$ denote the corresponding zero-initial-state classical
    solution. Then $(u_n,x_n,y_n)$ is also a generalised
    solution in the distributional sense. By Lemma~\ref{lem:L1-to-distributional-output-continuity}, the $\lpSpace{1}$ convergence implies that
    \[
        y_n(\varphi)\to y_u(\varphi)
        \quad
        \text{for every }
        \varphi\in \mathrm C_c^\infty((0,t);Y').
    \]

    Let $\psi\in \mathrm C_c^\infty((0,t))$ and $y'\in Y'$. Applying the distributional
    convergence to the test function $\psi y'$ gives
    \[
        \int_0^t \psi(\tau)\duality{y'}{y_u(\tau)}\dx \tau
        =
        \lim_{n\to\infty}
        \int_0^t \psi(\tau)\duality{y'}{y_n(\tau)}\dx \tau .
    \]
    Moreover, $\mathrm C^\infty$-BIBO stability ensures
    $
        \norm{y_n}_{\lpSpaceDT{\infty}{[0,t]}{Y}}
        \le
        c\norm{u_n}_{\lpSpaceDT{\infty}{[0,t]}{U}}
        \le
        cM
    $
    for all $n\in\bbN$.
    Hence,
    \begin{align*}
        \modulus{
            \int_0^t \psi(\tau)\duality{y'}{y_u(\tau)}\dx \tau
        }
        &\le
        \limsup_{n\to\infty}
        \int_0^t
        \modulus{\psi(\tau)}
        \norm{y'}_{Y'}
        \norm{y_n(\tau)}_Y
        \dx \tau                                    
        \le
        cM\norm{y'}_{Y'}
        \norm{\psi}_{\lpSpaceD{1}{[0,t]}} .
    \end{align*}
    Since $\duality{y'}{y_u(\argument)}\in \lpSpaceD{1}{[0,t]}$, this implies
    \begin{equation}
        \label{eq:distributional-estimate}
        \norm{\duality{y'}{y_u(\argument)}}_{\lpSpaceD{\infty}{[0,t]}}
        \le
        cM\norm{y'}_{Y'}
        \quad
        \text{for every }
        y'\in Y'.
    \end{equation}
    Since $y_u$ is represented by an element of $\lpSpaceDT{1}{[0,t]}{Y}$, it is strongly measurable and hence essentially separably
    valued. Thus there exists a separable closed subspace
    $Y_0\subseteq Y$ such that $y_u(\tau)\in Y_0$ for almost every $\tau\in(0,t)$.
    By separability, there exists a countable norming set
    $
        \{y_m'\}_{m\in\bbN}\subseteq B_{Y'}
    $
    for $Y_0$.
    By~\eqref{eq:distributional-estimate},
    \[
        \modulus{\duality{y_m'}{y_u(\tau)}}\le cM
        \quad
        \text{for almost every }
        \tau\in(0,t), m\in\bbN.
    \]
    Taking the intersection
    of the corresponding full-measure sets and using that
    $\{y_m'\}_{m\in\bbN}$ is norming for $Y_0$, we obtain
    $\norm{y_u(\tau)}_Y\le cM$ for almost every $\tau\in(0,t)$.
    Hence,
    $
        \norm{y_u}_{\lpSpaceDT{\infty}{[0,t]}{Y}}
        \le
        cM
        =
        c\norm{u}_{\lpSpaceDT{\infty}{[0,t]}{U}}
    $.
\end{proof}

Applying Proposition~\ref{prop:cinfty-bibo-existing-linfty-output-estimate} on every interval $(0,t)$ yields the following:

\begin{corollary}
    \label{cor:cinfty-bibo-plus-linfty-outputs-implies-linfty-bibo}
    Let $\systemNodeSt$ be a system node. Assume that
    $\systemNodeSt$ is $\mathrm C^\infty$-BIBO stable with constant $c\ge 0$.
    Assume moreover that, for every $u\in \lpSpace{\infty}_{\loc}(\bbR_+; U)$,
    the zero-initial-state generalised output $y_u$ is represented by an element of
    $\lpSpace{1}_{\loc}(\bbR_+; Y)$.
    Then $\systemNodeSt$ is $\lpSpace{\infty}$-BIBO stable with the same constant $c$.
\end{corollary}

Reflexivity of the output space ensures that the required output
representatives exist and hence yields an equivalence between
the two notions of BIBO stability.

\begin{theorem}
    \label{thm:reflexive-output-cinfty-implies-linfty-bibo}
    Let $\systemNodeSt$ be a $\mathrm C^\infty$-BIBO stable system
    node with constant $c\ge0$. If $Y$ is reflexive, then
    $\systemNodeSt$ is $\lpSpace{\infty}$-BIBO stable with the same
    constant $c$.
\end{theorem}

\begin{proof}
    Fix $u\in \lpSpace{\infty}_{\loc}(\bbR_+; U)$ and let $y_u$
    denote the corresponding zero-state generalised output.
    We show that $y_u$ is represented by an element of
    $\lpSpace{1}_{\loc}(\bbR_+; Y)$.
    Fix $t>0$ and put
    $M\coloneqq\norm{u}_{\lpSpaceDT{\infty}{[0,t]}{U}}$.
    Choose $u_n\in \mathrm C_c^\infty((0,t);U)$ such that
    \[
        u_n\to u
        \quad\text{in }\lpSpaceDT{1}{[0,t]}{U},
        \qquad
        \norm{u_n}_{\lpSpaceDT{\infty}{[0,t]}{U}}\le M.
    \]
    Let $y_n$
    be its zero-state classical output, hence it
    also represents the corresponding generalised output.
    The $\mathrm C^\infty$-BIBO estimate gives
    \[
        \norm{y_n}_{\lpSpaceDT{2}{[0,t]}{Y}}
        \le t^{1/2}\norm{y_n}_{\lpSpaceDT{\infty}{[0,t]}{Y}}
        \le t^{1/2}cM.
    \]
    Together with the reflexivity of $Y$, we obtain that $(y_n)$ is bounded in the reflexive space $\lpSpaceDT{2}{[0,t]}{Y}$.
    Passing to a subsequence, we may therefore assume that it converges weakly to some $z_t\in \lpSpaceDT{2}{[0,t]}{Y}$.

    Applying Lemma~\ref{lem:L1-to-distributional-output-continuity}
    to the zero extension of $u\restrict{(0,t)}$ and using
    causality gives
    $y_n\to y_u$ distributionally on $(0,t)$.
    On the other hand, for every
    $\varphi\in \mathrm C_c^\infty((0,t))$ and $y'\in Y'$,
    weak convergence gives
    \[
        \duality{y'}{y_u(\varphi)}
        =
        \lim_{n\to\infty}
        \int_0^t\varphi(s)\duality{y'}{y_n(s)}\dx s
        =
        \int_0^t\varphi(s)\duality{y'}{z_t(s)}\dx s.
    \]
    Since $Y'$ separates points of $Y$, $z_t$ represents
    $y_u$ on $(0,t)$.
    Moreover,
    $z_t\in \lpSpaceDT{2}{[0,t]}{Y}\subseteq \lpSpaceDT{1}{[0,t]}{Y}$.

    As $t>0$ was arbitrary, uniqueness of distributional
    representatives on overlapping intervals shows that
    $y_u$ is represented by an element of
    $\lpSpace{1}_{\loc}(\bbR_+; Y)$.
    Corollary~\ref{cor:cinfty-bibo-plus-linfty-outputs-implies-linfty-bibo}
    now yields the assertion.
\end{proof}

\section{Connection between BIBO stability and exponential stability}
\label{sec:BIBOandStability}

For finite-dimensional systems, exponential stability implies BIBO
stability; see, for instance,
\cite[Section~2.6.1, page~177]{b_Kailath80_LinearSystems}.
For infinite-dimensional system nodes, this implication fails in
general, even when the input and output spaces are finite-dimensional
\cite[Theorem~5.1]{SchwenningerWierzbaZwart2024}.
By Theorem~\ref{thm:finiteDimBIBOCharacterisation}, the counterexample
in that reference is therefore not even $\mathrm C^\infty$-BIBO stable.

Conversely, BIBO stability implies exponential stability for
controllable and observable finite-dimensional systems.
In infinite dimensions, the validity of this implication depends
on the precise notions of controllability and observability,
including the function spaces in which these properties are required.
Related implications for well-posed systems under exponential
stabilisability and detectability assumptions are given in
\cite[Theorem~8.2.11]{Staffans2005}.
Here we use finite-time null-controllability and final observability,
together with the corresponding admissibility assumptions.

For background on admissibility, controllability, and observability,
we refer to the monographs
\cite{Staffans2005,b_TucsnakWeiss2009}.
Let $p\in[1,\infty]$ and, for $t>0$, set
\begin{align*}
    \Phi_tu
    &\coloneqq\int_0^t\bbT_{-1}(t-s)Bu(s)\dx s,
    & u&\in\lpSpaceDT{p}{[0,t]}{U},\\
    (\Psi_tx)(s)
    &\coloneqq C\bbT(s)x,
    & x&\in X_1,\quad 0\le s\le t.
\end{align*}
Initially, $\Phi_t$ takes values in $X_{-1}$ and $\Psi_t$ is defined on $X_1$.
Whenever $\Psi_t$ has a bounded extension to $X$, we use the same
notation for this extension.
The relevant notions are summarised in
Table~\ref{tab:admissibility-controllability-observability}.
The admissibility conditions are required for some (equivalently, every) $t>0$.
The remaining conditions presuppose the corresponding admissibility.

\begin{center}
    \captionof{table}{Admissibility, controllability, and observability.}
    \label{tab:admissibility-controllability-observability}
    \renewcommand{\arraystretch}{1.2}
    \begin{tabularx}{\textwidth}{|
        >{\raggedright\arraybackslash}p{0.43\textwidth}
        @{\hspace{2em}}
        >{\raggedright\arraybackslash}X|}
        \hline
        Notion & Characterising condition \\
        \hline
        $\lpSpace{p}$-control-admissibility
        & $\Phi_t\in\calL(\lpSpaceDT{p}{[0,t]}{U},X)$. \\
        $\lpSpace{p}$-observation-admissibility
        & $\Psi_t$ extends to an element of
          $\calL(X,\lpSpaceDT{p}{[0,t]}{Y})$. \\
        \hline
        Infinite-time $\lpSpace{p}$-control-admissibility
        & $\sup_{t>0}\norm{\Phi_t}<\infty$. \\
        Infinite-time $\lpSpace{p}$-observation-admissibility
        & $\sup_{t>0}\norm{\Psi_t}<\infty$. \\
        \hline
        Exact $\lpSpace{p}$-controllability in time $t>0$
        & $\Ima\Phi_t=X$. \\
        Approximate $\lpSpace{p}$-controllability in time $t>0$
        & $\overline{\Ima\Phi_t}=X$. \\
        $\lpSpace{p}$-null-controllability in time $t>0$
        & $\Ima\bbT(t)\subseteq\Ima\Phi_t$. \\
        \hline
        Exact $\lpSpace{p}$-observability in time $t>0$
        & for some $c>0$:
          $\norm{x}
          \le c\norm{\Psi_t x}_p$
          for all $x\in X$. \\
        Approximate $\lpSpace{p}$-observability in time $t>0$
        & $\ker\Psi_t=\{0\}$. \\
        $\lpSpace{p}$-final observability in time $t>0$
        & for some $c>0$:
          $\norm{\bbT(t)x}
          \le c\norm{\Psi_t x}_p$
          for all $x\in X$. \\
        \hline
    \end{tabularx}
\end{center}
\medskip

The operator norms are taken between the spaces specified in the
admissibility rows.
A controllability or observability property holds in finite time
if it holds for some $t>0$.
Null-controllability in time $t$ means that every initial state can
be driven to zero at time $t$ by an $\lpSpace{p}$-control.

\begin{remark}
\label{rem:contrAndobsResults}
    Let $t>0$ and $p\in[1,\infty]$.
    If $B\in\calL(U,X)$ has dense range, then it is
    approximately controllable in time $t$ in $\lpSpace{p}$.
    Indeed, for $v\in U$ and integers $m>1/t$, set
    $
        u_m\coloneqq
        mv\one_{[t-1/m,t]}.
    $
    Strong continuity gives
    $
        \Phi_t u_m
        =
        m\int_0^{1/m}\bbT(r)Bv\dx r
        \to Bv
    $
    as $m\to \infty$,
    and hence
    $\overline{\Ima\Phi_t}\supseteq\overline{\Ima B}=X$.

    Similarly, if $C\in\calL(X,Y)$ is injective, then it is
    approximately observable in time $t$ in $\lpSpace{p}$.
    Indeed, if $\Psi_t x=0$ for some $x\in X$, then
    $C\bbT(s)x=0$ for almost every $s\in(0,t)$.
    By continuity, this holds for every $s\in[0,t]$.
    Taking $s=0$ gives $Cx=0$, and injectivity yields $x=0$.
\end{remark}

For the simplest case of identity input and output operators it was shown in~\cite[Lemma~3.3.14]{Wierzba2025} that van Neerven's convolution characterisation of exponential stability~\cite[Theorem~3.4.4]{vanNeerven1996b} gives an equivalence between exponential stability and BIBO stability. The following proposition shows that this result can be extended to a class of bounded input and output operators.

\begin{proposition}
    \label{prop:id-node-bibo-exp-stab}
    Consider a system node $\systemNodeSt$ with $B \in \calL(U,X)$ surjective and $C \in \calL(X,Y)$ bounded from below.
    Then the following assertions are equivalent:
    \begin{enumerate}[\upshape (i)]
        \item \label{enum:ExpStab}
        The semigroup $(\bbT(t))_{t\ge0}$ generated by $A$ is exponentially stable.
    
        \item \label{enum:LinftyBIBO}
        The system is $\lpSpace{\infty}$-BIBO stable.
    
        \item \label{enum:CinftyBIBO}
        The system is $\mathrm C^\infty$-BIBO stable.
    \end{enumerate}
\end{proposition}

The proof of this proposition utilises the following lemma.
\begin{lemma}
\label{lem:boundedBC-gen-solution-representation}
    Consider a system node $\systemNodeSt$ with $B \in \calL(U,X)$ and $C \in \calL(X,Y)$. Then for any zero-state generalised solution $(u,x,y)$ with $u \in \lpSpace{1}_{\loc}(\bbR_+; U)$ the distributional output lies in $\lpSpace{1}_{\loc}(\bbR_+; Y)$ and is represented by
    \[
        y(t) = C \int_0^t \bbT(t-s) B u(s) \dx s + D u(t)
        \quad\text{for almost every }t\ge0.
    \]
    Here the feedthrough operator
    $D \coloneqq \mathbf{G}(\lambda) - C \resolvent{\lambda}{A} B \in \calL(U,Y)$
    is independent of
    $\lambda\in\bbC_{\gbd(\bbT)}$.
\end{lemma}

\begin{proof}
    Independence of $\lambda$ follows from
    \eqref{eq:transferFunctionDifferenceRelation}.
    Since $B\in\calL(U,X)$, 
    the mild state
    $
        x(t)=\int_0^t\bbT(t-s)Bu(s)\dx s
    $
    belongs to $\mathrm C(\bbR_+;X)$.
    Together with boundedness of $C$ and $D$, this gives
    $Cx+Du\in \lpSpace{1}_{\loc}(\bbR_+; Y)$.
    Moreover, $B\in\calL(U,X)$ also implies that  
    $\dom{C\&D}=\dom{A}\times U$ with
    \[
        C\&D\begin{bmatrix}z\\v\end{bmatrix}
        =Cz+Dv
        \quad
        \text{for }
        \begin{bmatrix}z\\v\end{bmatrix}
        \in\dom{C\&D}.
    \] 
    Applying this identity to the integrated pair in the
    definition of the generalised output yields
    \[
        F(t)
        =
        (C\&D)\int_0^t(t-s)
            \begin{bmatrix}x(s)\\u(s)\end{bmatrix}\dx s
        =
        \int_0^t(t-s)\bigl(Cx(s)+Du(s)\bigr)\dx s.
    \]
    Taking two distributional derivatives gives
    $y=F''=Cx+Du$.
\end{proof}

\begin{proof}[Proof of Proposition~\ref{prop:id-node-bibo-exp-stab}]
    Denote the zero-state input-output map of the system, the form of which follows from Lemma~\ref{lem:boundedBC-gen-solution-representation}, by
    \[
        (\Gamma u)(t)\coloneqq C \int_0^t \bbT(t-s)Bu(s)\dx s + D u(t).
    \]

    \Implies{enum:ExpStab}{enum:CinftyBIBO}
    Let $M\ge1$ and $\omega>0$ such that
    $\norm{\bbT(t)}\le M e^{-\omega t}$ for all $t\ge 0$.
    Let $u\in \mathrm C_c^\infty((0,\infty);U)$. For each $t>0$, we have
    \[
    \begin{aligned}
        \norm{\Gamma u}_{\lpSpaceDT{\infty}{[0,t]}{Y}}
        &=
        \sup_{0\le r\le t}
         \norm{
            C\int_0^r \bbT(r-s) Bu(s)\dx s + Du(r)
        } \\
        &\le
        M \norm{B}_{\calL(U,X)} \norm{C}_{\calL(X,Y)} \norm{u}_{\lpSpaceDT{\infty}{[0,t]}{U}}
        \sup_{0\le r\le t}
        \int_0^r e^{-\omega(r-s)}\dx s
        +
        \norm{D}
        \norm{u}_{\lpSpaceDT{\infty}{[0,t]}{U}}
        \\
        &\le
        \bigl(M \omega^{-1} \norm{B}_{\calL(U,X)} \norm{C}_{\calL(X,Y)}+ \norm{D}\bigr)
        \norm{u}_{\lpSpaceDT{\infty}{[0,t]}{U}} .
    \end{aligned}
    \]
    Thus the system is $\mathrm C^\infty$-BIBO stable.
    
    \Implies{enum:CinftyBIBO}{enum:LinftyBIBO}
    Since $B$ and $C$ are bounded, Lemma~\ref{lem:boundedBC-gen-solution-representation} shows that for every $u\in \lpSpace{\infty}_{\loc}(\bbR_+; U)$, the
    corresponding generalised output is represented by
    an element of $\lpSpace{1}_{\loc}(\bbR_+; Y)$.
    Corollary~\ref{cor:cinfty-bibo-plus-linfty-outputs-implies-linfty-bibo} thus
    shows that $\mathrm C^\infty$-BIBO stability imples $\lpSpace{\infty}$-BIBO stability.
    
    \Implies{enum:LinftyBIBO}{enum:ExpStab}
    Assume that the system is $\lpSpace{\infty}$-BIBO stable.
    Since adding or subtracting the bounded feedthrough term $Du$ does not
    affect $\lpSpace{\infty}$-BIBO stability, the system
    $
        \systemNode
            {A}
            {B}
            {C}
            {C\resolvent{\argument}{A}B}
    $
    is also $\lpSpace{\infty}$-BIBO stable. Let $c_{\mathrm{bibo}}>0$ be a BIBO
    constant for this system.

    Since $B:U\to X$ is bounded and surjective, the Bartle-Graves theorem \cite[Proposition~1.19(ii)]{BenyaminiLindenstrauss2000}
    yields a continuous map
    $
        S:X\to U
    $
    and a constant $K>0$ such that
    \[
        B\circ S=\id_X 
        \quad\text{and}\quad
        \norm{Sx}_U\le K\norm{x}_X
        \quad
        \text{for all }x\in X.
    \]
    Let $f\in \lpSpace{\infty}_{\loc}(\bbR_+; X)$ and set
    $
        u(t)\coloneqq S(f(t)).
    $
    Since $S$ is continuous, $u$ is strongly measurable with
    \[
        Bu(t)=f(t)
        \quad\text{and}\quad
        \norm{u}_{\lpSpaceDT{\infty}{[0,t]}{U}}
        \le
        K\norm{f}_{\lpSpaceDT{\infty}{[0,t]}{X}}.
    \]
    Thus,
    the fact that $C$ is bounded below together with the choice of $c_{\mathrm{bibo}}$ gives
    that
    \[
        \norm{\int_0^{\argument}\bbT(\argument-s)f(s)\dx s}_{\lpSpaceDT{\infty}{[0,t]}{X}}
        \le
        \delta^{-1}\norm{C\int_0^{\argument}\bbT(\argument-s)f(s)\dx s}_{\lpSpaceDT{\infty}{[0,t]}{Y}}
        \le
        c_{\mathrm{bibo}}K\delta^{-1}
        \norm{f}_{\lpSpaceDT{\infty}{[0,t]}{X}}
    \]
    for some $\delta>0$.
    Hence the system node
    $
        \systemNode
            {A}
            {\id_X}
            {\id_X}
            {\resolvent{\argument}{A}}
    $
    is $\lpSpace{\infty}$-BIBO stable. In turn, $(\bbT(t))_{t\ge0}$ is exponentially stable
    by~\cite[Lemma~3.3.14]{Wierzba2025}.
\end{proof}

Approximate controllability and observability do not suffice
to recover the converse implication.
The following examples exhibit $\lpSpace{\infty}$-BIBO stable systems
whose semigroups are not strongly stable, and can even be unbounded.

\begin{example}[$\lpSpace{\infty}$-BIBO stability without strong stability]
\label{ex:simple-counterex-approx-obs-cont-exp-stable}
    Let $X=U=Y=\lpSpaceD{2}{\bbR_+}$, and let $A$ generate
    the right-shift semigroup $\bbT$ on $X$.
    Set $B=\id_X$ and define $C\in\calL(X,Y)$ by
    \[
        (Cf)(\xi)=e^{-\xi}f(\xi)
        \qquad (\xi\ge0).
    \]
    Consider the system node with transfer function
    $
        \mathbf{G}(\argument)=C\resolvent{\argument}{A}B
    $.
    Thus the feedthrough operator in
    Lemma~\ref{lem:boundedBC-gen-solution-representation}
    is zero.
    The semigroup is isometric and hence not strongly stable.
    Since $B$ has dense range and $C$ is injective,
    Remark~\ref{rem:contrAndobsResults} gives approximate
    controllability and approximate observability in every
    time $t>0$ and for each $p\in[1,\infty]$.

    We next show that the corresponding system node is $\lpSpace{\infty}$-BIBO stable.
    For $f\in X$ and $t\ge 0$, we have
    $
        (C\bbT(t)f)(\xi)
        =
        e^{-\xi} f(\xi-t)\one_{[t,\infty)}(\xi)
    $.
    Therefore
    \begin{align*}
        \norm{C\bbT(t)f}_{\lpSpaceD{2}{\bbR_+}}^2
        =
        \int_t^\infty e^{-2\xi}\modulus{f(\xi-t)}^2\dx{\xi} 
        =
        e^{-2t}\int_0^\infty e^{-2\eta}\modulus{f(\eta)}^2\dx{\eta} 
        \le
        e^{-2t}\norm{f}_{\lpSpaceD{2}{\bbR_+}}^2 .
    \end{align*}
    Thus
    $
        \int_0^\infty \norm{C\bbT(t)B}\dx{t}
        \le
        \int_0^\infty e^{-t}\dx{t}
        =
        1.
    $    
    Hence, for every $u\in \lpSpace{\infty}_{\loc}(\bbR_+; U)$ and zero initial
    state, the corresponding output satisfies
    \[
        \norm{y}_{\lpSpaceDT{\infty}{[0,t]}{Y}}
        \le
        \norm{u}_{\lpSpaceDT{\infty}{[0,t]}{U}}
        \int_0^\infty \norm{C\bbT(s)B}\dx{s}
        \le
        \norm{u}_{\lpSpaceDT{\infty}{[0,t]}{U}}
    \]
    for all $t\ge 0$. Thus $\systemNodeSt$ is $\lpSpace{\infty}$-BIBO stable.
\end{example}

\begin{remark}
    Choose instead $X=U=Y=\lpSpaceD{2}{\bbR_+,(1+\xi)\dxShort\xi}$ in Example~\ref{ex:simple-counterex-approx-obs-cont-exp-stable}.
    Then
    \[
        \norm{\bbT(t)}=(1+t)^{1/2}
        \qquad\text{and}\qquad
        \norm{C\bbT(t)B}
        \le e^{-t}(1+t)^{1/2}.
    \]
    Thus the semigroup is unbounded, whereas the convolution
    estimate in that example still proves $\lpSpace{\infty}$-BIBO stability.
    Approximate controllability and approximate observability
    again follow from Remark~\ref{rem:contrAndobsResults}.
\end{remark}

Note that the semigroup in Example~\ref{ex:simple-counterex-approx-obs-cont-exp-stable} is not analytic. However, even for analytic semigroups, counterexamples can be found.

\begin{example}
[$\lpSpace{\infty}$-BIBO stability with an unbounded analytic semigroup]
\label{ex:analytic-counterex-approx-obs-cont-not-strongly-stable}
    Let $X=U=Y=\lpSpaceDT{2}{[0,1]}{\bbC^2}$
    and define
    \[
        (Af)(\xi)
        \coloneqq
        \begin{bmatrix}
            -\xi & 1\\
            0 & -\xi
        \end{bmatrix}f(\xi),
        \qquad
        (Bf)(\xi)\coloneqq \xi^3f(\xi),
        \qquad
        C\coloneqq\id_X.
    \]
    Consider the system node with transfer function
    $
        \mathbf{G}(\argument)=C\resolvent{\argument}{A}B
    $.
    Since $A\in\calL(X)$, it generates an analytic semigroup given by
    \[
        (\bbT(t)f)(\xi)
        =
        e^{-t\xi}
        \begin{bmatrix}
            1&t\\
            0&1
        \end{bmatrix}
        f(\xi),
        \qquad t\ge0.
    \]
    Moreover,
    $
        \norm{\bbT(t)}
        \ge
        \norm{t e^{-t\argument}}_{\lpSpaceD{\infty}{[0,1]}}
        =t
    $,
    so $(\bbT(t))_{t\ge0}$ is unbounded. Moreover, as $B$ has dense range in $X$, Remark~\ref{rem:contrAndobsResults} ensures approximate controllability for every time $t>0$ and each $p\in [1,\infty]$.

    On the other hand, by uniform continuity, there exists
    $\delta\in(0,t)$ such that
    $\norm{\bbT(s)-\id_X}_{\calL(X)}<2^{-1}$
    for every $0\le s\le\delta$.
    Consequently,
    \[
        \norm{\bbT(s)z}_X
        \ge
        \norm{z}_X-\norm{(\bbT(s)-\id_X)z}_X
        \ge 2^{-1}\norm{z}_X
    \]
    for every $z\in X$ and $0\le s\le\delta$.
    Since $C=\id_X$, it follows that
    \[
        \norm{\Psi_t z}_{\lpSpaceDT{q}{[0,t]}{Y}}
        \ge
        \begin{cases}
            2^{-1}\delta^{1/q}\norm{z}_X, \qquad & q<\infty\\
            2^{-1}\norm{z}_X,                \qquad & q=\infty
        \end{cases}
    \]
    for each $z\in X$.
    Thus the system is finite-time exactly $\lpSpace{q}$-observable for every
    $q\in[1,\infty]$.

    The impulse response is the multiplication operator given by
    \[
        \bigl(C\bbT(t)Bf\bigr)(\xi)
        =
        \xi^3e^{-t\xi}
        \begin{bmatrix}
            1&t\\
            0&1
        \end{bmatrix}f(\xi).
    \]
    Using
    $
        \sup_{\xi\in[0,1]}\xi^3e^{-t\xi}
        \le
        \min\{1,27e^{-3}t^{-3}\}
        \lesssim
        (1+t)^{-3}
    $,
    we obtain
    \[
        \norm{C\bbT(t)B}_{\calL(U,Y)}
        \lesssim
        (1+t)
        \sup_{\xi\in[0,1]}\xi^3e^{-t\xi}
        \lesssim
        (1+t)^{-2}.
    \]
    Thus
    $
        C\bbT(\argument)B
        \in \lpSpaceDT{1}{\bbR_+}{\calL(U,Y)}
    $.
    Since $B$ and $C$ are bounded and
    $\mathbf{G}(s)=C\resolvent{s}{A}B$,
    Lemma~\ref{lem:boundedBC-gen-solution-representation}
    and Young's inequality imply that the system node
    is $\lpSpace{\infty}$-BIBO stable.
\end{example}

\begin{proposition}
\label{prop:null-cont-final-obs-exponStab}
    Let $\systemNodeSt$ be an $\lpSpace{\infty}$-BIBO stable system node. Assume that
    \begin{enumerate}[\upshape (a)]
        \item $B$ is $\lpSpace{\infty}$-null-controllable in
        finite time, and
        \item $C$ is
        $\lpSpace{r}$-finally-observable in finite time for some $r\in[1,\infty]$.
    \end{enumerate}
    Then the semigroup $(\bbT(t))_{t\ge0}$ generated by $A$ is exponentially
    stable.
\end{proposition}

\begin{proof}
    For $t>0$, denote the admissible control and observation maps by
    \[
        \Phi_t : \lpSpaceDT{\infty}{[0,t]}{U} \to X
        \quad\text{and}\quad
        \Psi_t : X\to \lpSpaceDT{r}{[0,t]}{Y}
    \]
    respectively.
    Controls on $[0,t]$ are extended by zero when shifted below.
    By increasing the controllability and observability times if necessary, we
    may choose $t>0$ such that 
    $
        \Ima\bbT(t)
            \subseteq
        \Ima\Phi_t
    $
    and there exists $c_{\mathrm{obs}}>0$ such that
    $\norm{\Psi_t z}_{\lpSpaceDT{r}{[0,t]}{Y}}
        \ge
        c_{\mathrm{obs}}\norm{\bbT(t)z}_X$
    for all $z\in X$.
    Since 
    $
        \Ima\bbT(t)
            \subseteq
        \Ima\Phi_t
    $,
    the operator
    \[
        \Phi_t^{-1}\bbT(t)
        :
        X
        \longrightarrow
        \lpSpaceDT{\infty}{[0,t]}{U}/\ker\Phi_t
    \]
    is bounded by~\cite[Lemma~9.10.10]{Staffans2005}. Hence there exists $K>0$ such that
    \[
        \norm{
            \Phi_t^{-1}\bbT(t)x
        }_{\lpSpaceDT{\infty}{[0,t]}{U}/\ker\Phi_t}
        \le
        \frac{K}{2}\norm{x}_X
        \qquad
        \text{for all }x\in X.
    \]
    By the definition of the quotient norm, for every $x\in X$ one can
    therefore choose $v_x\in \lpSpaceDT{\infty}{[0,t]}{U}$ with $\norm{v_x}_{\lpSpaceDT{\infty}{[0,t]}{U}}\le K\norm{x}_X$ and $\Phi_t v_x=\bbT(t)x$.

    Fix now $\lambda\in\bbC$ with $\modulus{\lambda}=1$ and $n\in\bbN$.
    For $x\in X$, choose $v_x$ as above and define
    \[
        u_n
        \coloneqq
        \sum_{j=0}^{n-1}
        \overline{\lambda}^{\,j}
        v_x(\argument-jt)
        \one_{[jt,(j+1)t)}
        \in \lpSpaceDT{\infty}{\bbR_+}{U}
        .
    \]
    Then $\norm{u_n}_{\lpSpaceDT{\infty}{\bbR_+}{U}}\le K\norm{x}_X$.
    Let $(u_n,x_n,y_n)$ be the corresponding generalised solution with zero
    initial state. Then
    \begin{equation}
        \label{eq:power-sums}
        \begin{split}
            x_n(nt)
            &=\sum_{j=0}^{n-1} \int_{jt}^{(j+1)t}\bbT_{-1}(nt-s)Bu_n(s)\dx s
            =\sum_{j=0}^{n-1}\overline{\lambda}^{\,j} \int_{jt}^{(j+1)t}\bbT_{-1}(nt-s)B v_x(s-jt)\dx s\\
            &=\sum_{j=0}^{n-1}\overline{\lambda}^{\,j} \int_{0}^{t}\bbT_{-1}((n-j)t-s)B v_x(s)\dx s
            =\sum_{j=0}^{n-1}\overline{\lambda}^{\,j}\bbT( (n-j-1)t)\Phi_{t}v_x\\
            &=\sum_{j=0}^{n-1} 
            \overline{\lambda}^{\,j}
            \bbT(t)^{n-1-j}\Phi_t v_x
            =
            \sum_{j=0}^{n-1}
            \overline{\lambda}^{\,j}
            \bbT(t)^{n-j}x
            =
            \overline{\lambda}^{\,n-1}
            \bbT(t)\sum_{k=0}^{n-1}\lambda^{k}\bbT(t)^k x;
        \end{split}
    \end{equation}
    where we have used $\overline{\lambda}=\lambda^{-1}$.
    Since $u_n$ vanishes on $[nt,\infty)$, for each $s\ge 0$,
    \[
        x_n(nt+s)=\int_0^{nt} \bbT_{-1}(nt+s-r)Bu_n(r)\dx r+\int_{nt}^{nt+s} \bbT_{-1}(nt+s-r)Bu_n(r)\dx r
                =\bbT(s)x_n(nt).
    \]
    
    By time invariance of the system node, the shifted triple $(u_n(nt+\argument), x_n(nt+\argument), y_n(nt+\argument))$ is a
    generalised solution with input $u_n(nt+\argument)=0$ and initial state $x_n(nt)$. 
    By $\lpSpace{r}$-observation-admissibility, the shifted output is represented by
    $
        y_n(nt+\argument)
        =
        \Psi_t x_n(nt)
    $
    in $\lpSpaceDT{r}{[0,t]}{Y}$.
    Set
    \[
        c_{t,r}
        \coloneqq
        \begin{cases}
            t^{1/r}, & r<\infty,\\
            1,          & r=\infty
        \end{cases}
    \]
    and
    let $c_{\mathrm{bibo}}>0$ be a BIBO constant.
    By final observability,
    \begin{align*}
        c_{\mathrm{obs}}
        \norm{\bbT(t)x_n(nt)}_X
        &\le
        \norm{y_n}_{\lpSpaceDT{r}{[nt,(n+1)t]}{Y}}
        \le
        c_{t,r}
        \norm{y_n}_{\lpSpaceDT{\infty}{[nt,(n+1)t]}{Y}}\\
        &\le
        c_{t,r}c_{\mathrm{bibo}}
        \norm{u_n}_{\lpSpaceDT{\infty}{\bbR_+}{U}}
        \le
        c_{t,r}c_{\mathrm{bibo}}K\norm{x}_X.
    \end{align*}
    Substituting in~\eqref{eq:power-sums} and using
    $\modulus{\lambda}=1$, we obtain
    \[
        \norm{
            \bbT(t)^2
            \sum_{k=0}^{n-1}\lambda^k\bbT(t)^k x
        }_X
        \le
        \frac{c_{t,r}c_{\mathrm{bibo}}K}
             {c_{\mathrm{obs}}}
        \norm{x}_X.
    \]
    In particular,
    \[
        \sup_{n\in\bbN}
        \norm{
            \bbT(t)^2\sum_{k=0}^{n-1}\lambda^k\bbT(t)^k
        }_{\calL(X)}
        <\infty.
    \]
    Combining with
    \begin{align*}
        \sum_{k=0}^{n+1}\lambda^k\bbT(t)^k
        =
        \id_X+\lambda \bbT(t)
        +
        \sum_{k=2}^{n+1}\lambda^k\bbT(t)^k
        =
        \id_X+\lambda \bbT(t)
        +
        \lambda^2\bbT(t)^2
        \sum_{k=0}^{n-1}\lambda^k\bbT(t)^k,
    \end{align*}
    it follows that for every
    $\lambda\in\bbC$ with $\modulus{\lambda}=1$, there exists 
    $c_\lambda>0$ such that
    $
        \norm{
            \sum_{k=0}^{n}\lambda^k\bbT(t)^k
        }_{\calL(X)}
        \le c_\lambda
    $
    for all $n\in\bbN_0$.
    Lemma~\ref{lem:power-sum-spectrum}\ref{lem:power-sum-spectrum:itm:all-phases} therefore yields $\spr(\bbT(t))<1$, and in turn, exponential stability of $(\bbT(\tau))_{\tau\ge0}$.
\end{proof}

The preceding proof used the following spectral theoretic lemma.

\begin{lemma}
    \label{lem:power-sum-spectrum}
    Let $T$ be a bounded linear operator on a Banach space $X$ with spectral radius $\spr(T)$. 
    \begin{enumerate}[\upshape(a)]
        \item\label{lem:power-sum-spectrum:itm:positive} 
        If $\sup_{n\in\bbN_0} \norm{\sum_{k=0}^n T^k} <\infty$, then $[1,\infty)\subseteq \resSet(T)$.

        \item\label{lem:power-sum-spectrum:itm:all-phases} 
        If $\sup_{n\in\bbN_0}\norm{\sum_{k=0}^n \lambda^k T^k} <\infty $ for every $\lambda \in \bbC$ with $\modulus{\lambda} = 1$,
        then $\spr(T) < 1$.
    \end{enumerate}
\end{lemma}

\begin{proof}
    \ref{lem:power-sum-spectrum:itm:positive}
    Assume that $\mu \in [1,\infty)$ were a spectral value of $T$. 
    By the spectral mapping theorem for polynomials, the number $\sum_{k=0}^n \mu^k$ is then a spectral value of $\sum_{k=0}^n T^k$ for each $n \in \bbN_0$. 
    Hence, 
    \begin{align*}
        n+1 
        \le 
        \sum_{k=0}^n \mu^k 
        \le 
        \spr\Big(\sum_{k=0}^n T^k\Big) 
        \le 
        \norm{\sum_{k=0}^n T^k} 
        \le 
        \sup_{m\in\bbN_0} \norm{\sum_{k=0}^m T^k}
    \end{align*}
    for all $n \in \bbN_0$, which is a contradiction.

    \ref{lem:power-sum-spectrum:itm:all-phases}
    It follows from the assumption and from~\ref{lem:power-sum-spectrum:itm:positive} that every $\mu \in \bbC$ with $\modulus{\mu} \ge 1$ is in the resolvent set of $T$. 
    Since the spectrum of $T$ is compact we conclude that $\spr(T) < 1$.
\end{proof}

The proof of Proposition~\ref{prop:null-cont-final-obs-exponStab} uses $\lpSpace{\infty}$-BIBO stability to estimate the output on
a fixed observation interval after the input has been switched off. For a
global $\lpSpace{p}$-to-$\lpSpace{q}$ input-output property with $q<\infty$, one can instead
sum the output estimates over infinitely many consecutive observation
intervals. This yields a discrete Datko estimate and leads to the following
analogous result.

\begin{proposition}
\label{prop:lp-lq-io-exp-stab}
    Let $p\in[1,\infty]$ and $q\in[1,\infty)$. 
    Let $\systemNodeSt$ be a system node such that $B$ is finite-time $\lpSpace{p}$-null-controllable, and $C$ is finite-time $\lpSpace{q}$-finally-observable.
    
    In addition, assume that, for every $u\in \lpSpaceDT{p}{\bbR_+}{U}$, the
    zero-initial-state generalised output is represented by an element of
    $\lpSpaceDT{q}{\bbR_+}{Y}$.
    Then the semigroup generated by $A$ is exponentially stable.
\end{proposition}

\begin{proof}
    First observe that the zero-state input-output map
    \[
        \Gamma:
        \lpSpaceDT{p}{\bbR_+}{U}
        \longrightarrow
        \lpSpaceDT{q}{\bbR_+}{Y},
        \qquad
        u\longmapsto y_u,
    \]
    is automatically bounded. Indeed, suppose that
    $
        u_n\to u
    $
    in $\lpSpaceDT{p}{\bbR_+}{U}$ and
    $
        \Gamma u_n\to y
    $
    in $\lpSpaceDT{q}{\bbR_+}{Y}$.
    For every $t>0$ we have
    $
        u_n\to u
    $
    in $\lpSpaceDT{1}{[0,t]}{U}$.
    Applying Lemma~\ref{lem:L1-to-distributional-output-continuity}
    to the zero extensions of
    $u_n\restrict{(0,t)}$ and $u\restrict{(0,t)}$,
    and using causality, gives
    $
        \Gamma u_n\longrightarrow\Gamma u
    $
    distributionally on $(0,t)$. 
    On the other hand,
    $
        \Gamma u_n\to y
    $
    in $\lpSpaceDT{q}{[0,t]}{Y}$ and therefore also distributionally on $(0,t)$.
    Thus $y=\Gamma u$ on $(0,t)$. Since $t>0$ was arbitrary, $\Gamma$ is closed, hence bounded. Thus, there exists
    $\kappa>0$ such that
    \[
        \norm{y_u}_{\lpSpaceDT{q}{\bbR_+}{Y}}
        \le
        \kappa\norm{u}_{\lpSpaceDT{p}{\bbR_+}{U}}
        \qquad
        \text{for all }
        u\in \lpSpaceDT{p}{\bbR_+}{U}.
    \]

    As in the proof of Proposition~\ref{prop:null-cont-final-obs-exponStab},
    by increasing the controllability and observability times if necessary,
    we may choose $t>0$ such that
    $\Ima\bbT(t)\subseteq\Ima\Phi_t$ and there exists
    $\delta>0$ with
    $\norm{\Psi_t z}_{\lpSpaceDT{q}{[0,t]}{Y}}
    \ge \delta\norm{\bbT(t)z}_X$ for every $z\in X$.
    The same quotient-space argument also yields a constant $K>0$ such that,
    for every $x\in X$, there exists
    $u_x\in \lpSpaceDT{p}{[0,t]}{U}$ satisfying
    $\Phi_t u_x=\bbT(t)x$ and
    $\norm{u_x}_{\lpSpaceDT{p}{[0,t]}{U}}\le K\norm{x}_X$.

    Extend $u_x$ by zero to $\bbR_+$ and let $(u_x,x_x,y_x)$ be the
    corresponding generalised solution with zero initial state. By assumption,
    $\norm{y_x}_{\lpSpaceDT{q}{\bbR_+}{Y}}
    \le \kappa K\norm{x}_X$.
    Since $u_x$ vanishes on $[t,\infty)$,
    time invariance and $\lpSpace{q}$-admissibility of $C$ imply that, for every
    $k\in\bbN_0$, the restriction of $y_x$ to
    $[(k+1)t,(k+2)t]$ is the zero-input output corresponding to
    the initial state $\bbT(t)^{k+1}x$. Hence final observability gives
    \[
        \kappa^q K^q\norm{x}_X^q
        \ge
        \norm{y_x}_{\lpSpaceDT{q}{\bbR_+}{Y}}^q
        \ge
        \delta^q\sum_{k=0}^\infty\norm{\bbT(t)^{k+2}x}_X^q .
    \]
    So, there exists $M>0$ such that
    $\sum_{k=0}^\infty\norm{\bbT(t)^k x}_X^q\le M\norm{x}_X^q$ for every
    $x\in X$. Since $q<\infty$, the discrete Datko theorem \cite[Proposition~1.2]{Weiss1989d} yields $\spr(\bbT(t))<1$, and therefore
    exponential stability of $\bbT$.
\end{proof}

\section{BIBO stability of derived single-input single-output systems}   
\label{sec:BIBOSISO}

Consider a system node $\systemNodeSt$ on $(U,X,Y)$.
For $u_0\in U$ and $y_0\in Y'$, restricting the input to
$v\mapsto u_0v$ and scalarising the output by $y_0$ gives the
single-input single-output (SISO) system node
$\systemNode{A}{Bu_0}{y_0C}{y_0\mathbf{G} u_0}$
on $(\bbC,X,\bbC)$; see Figure~\ref{fig:derivedSISOsystems}.
If $\systemNodeSt$ is $\mathrm C^\infty$-BIBO stable or
$\lpSpace{\infty}$-BIBO stable with constant $c$, then each derived
SISO system has that property with constant
$c\norm{u_0}_U\norm{y_0}_{Y'}$.
In particular, $c$ is a common BIBO constant for all normalised
directions $\norm{u_0}_U=\norm{y_0}_{Y'}=1$.

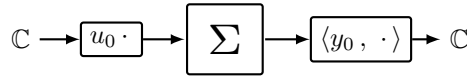
\begin{figure}[h!]
    \centering
    \tikzstyle{block} = [draw, rectangle, rounded corners=1, thick]
    \tikzstyle{empty} = [minimum height=1em, minimum width=1em]
    \begin{tikzpicture}[auto, node distance=1.5cm,>=]
    \node [empty, name=input] {$\bbC$};
    \node [block, right of=input, node distance=1.2cm] (scInput) {$u_0\argument$};
    \node [block, right of=scInput, node distance=1.5cm, scale=2.3] (system) {$\scriptstyle\Sigma$};   
    \node [block, right of=system, node distance=1.8cm] (scOutput) {$\duality{y_0}{\argument}$};
    \node [empty, right of=scOutput, node distance=1.3cm] (output) {$\bbC$}; 
    \draw [-latex,thick] (input) -- (scInput);
    \draw [-latex,thick] (scInput) -- (system);
    \draw [-latex,thick] (system) -- (scOutput);
    \draw [-latex,thick] (scOutput) -- (output);
    \end{tikzpicture}
    \caption{SISO systems derived from a general system node $\Sigma$}
    \label{fig:derivedSISOsystems}
\end{figure}

Then, conversely, does uniform BIBO stability of all normalised derived SISO
systems imply BIBO stability of the full system?
Such an implication would provide a useful criterion, since
Theorem~\ref{thm:finiteDimBIBOCharacterisation} characterises
BIBO stability of the derived systems by scalar measures of
bounded total variation.
The following example shows that the full system need not
even be $\mathrm C^\infty$-BIBO stable, even if both control
and observation operators are bounded.

\begin{example}[Uniform SISO BIBO stability without BIBO stability]
    \label{exa:siso-bibo-bounded-B-C}
    Consider $U = X = \lpSpaceD{1}{\bbR}$ and $Y = \bbC$. Let $A$ be the generator of the right shift semigroup $\bbT$ on $X$. Set $B = \id$ and
    $Cf=\int_0^1 f(\xi)\dx\xi$, and choose
    $\mathbf{G}(\argument)=C\resolvent{\argument}{A}{B}$.
    By Lemma~\ref{lem:boundedBC-gen-solution-representation},
    the feedthrough operator is zero.
    
    Fix $u_0\in U$ and
    $u\in\mathrm C_c^\infty((0,\infty))$.
    The zero-state output corresponding to the input $u_0u$
    satisfies, by Tonelli's theorem,
    \begin{align*}
        \modulus{y(t)}
        \le
        \norm{u}_{\lpSpaceD{\infty}{[0,t]}}
        \int_0^\infty\int_0^1
            \modulus{u_0(\xi-s)}\dx\xi\dx s
        \le
        \norm{u}_{\lpSpaceD{\infty}{[0,t]}}\norm{u_0}_U,
        \qquad (t>0).
    \end{align*}
    Thus $\systemNode{A}{Bu_0}{C}{\mathbf{G} u_0}$
    is $\mathrm C^\infty$-BIBO stable with constant
    $\norm{u_0}_U$, and hence also $\lpSpace{\infty}$-BIBO stable by
    Theorem~\ref{thm:finiteDimBIBOCharacterisation}.
    Further scalarisation of the output multiplies this bound
    by $\norm{y_0}_{Y'}$.
    Consequently, $1$ is a common BIBO constant for all
    normalised derived SISO systems.

    To show that the full system is not
    $\mathrm C^\infty$-BIBO stable, use the isometric group
    extension of $\bbT$ to negative times.
    Fix $\varphi\in\mathrm C_c^\infty(\bbR;[0,1])$
    with $\varphi=1$ on $[0,1]$.
    For each $t>1$, choose
    $\rho_t\in\mathrm C_c^\infty((0,\infty);[0,1])$
    with $\rho_t=1$ on $[1,t]$, and define
    $u_t(s)=\rho_t(s)\bbT(s-t)\varphi$ for $s\ge0$.
    Since $\varphi\in\dom{A^m}$ for every $m\in\bbN$,
    we have $u_t\in\mathrm C_c^\infty((0,\infty);U)$.
    Moreover,
    $\norm{u_t}_{\lpSpaceDT{\infty}{\bbR_+}{U}}=\norm{\varphi}_U$,
    independently of $t$.
    Since $C\varphi=1$, the corresponding zero-state output
    satisfies
    \[
        y_{u_t}(t)
        =\int_0^t\rho_t(s)C\varphi\dx s
        =\int_0^t\rho_t(s)\dx s
        \ge t-1.
    \]
    This rules out $\mathrm C^\infty$-BIBO stability. 
\end{example}

One may ask whether exponential stability restores this
implication.
If $B\in\calL(U,X)$ and $C\in\calL(X,Y)$, exponential
stability already implies $\lpSpace{\infty}$-BIBO stability by
Lemma~\ref{lem:boundedBC-gen-solution-representation}
and the usual convolution estimate.
The next example shows that the implication fails for
unbounded control and observation operators, even on Hilbert
spaces and with infinite-time $\lpSpace{2}$-admissibility.

\begin{example}[Uniform SISO BIBO stability without BIBO stability despite exponential stability on a Hilbert space]
\label{exa:siso-bibo-unbounded-B-C}
    Let $U=X=Y=\ell^2$, with canonical basis $(e_n)_{n\ge1}$.
    Set $\lambda_n=2^n$ and define
    \[
        Ae_n=-\lambda_ne_n,
        \qquad
        \dom{A}
        =
        \left\{
            x\in\ell^2:
            (\lambda_nx_n)_{n\ge1}\in\ell^2
        \right\}.
    \]
    Then $A$ is self-adjoint and generates the exponentially
    stable analytic semigroup
    $\bbT(t)e_n=e^{-\lambda_nt}e_n$.

    The control and observation operators $B\coloneqq (-A_{-1})^{1/2}\in \calL(X, X_{-1})$ and $C\coloneqq (-A)^{1/2}\in \calL(X_1,X)$ are infinite-time $\lpSpace{2}$-admissible since
    \[
        \norm{
            \int_0^t\bbT_{-1}(t-s)Bu(s)\dx s
        }_X^2
        \le
        \frac12\int_0^t\norm{u(s)}_X^2\dx s,
        \qquad
        \int_0^\infty\norm{C\bbT(t)x}_X^2\dx t
        =\frac12\norm{x}_X^2.
    \]
    The control estimate follows by applying the
    Cauchy--Schwarz inequality in each coordinate.
    Choose
    $\mathbf{G}(z)\coloneqq(-A)\resolvent{z}{A}\in\calL(X)$
    for $\re z>-2$.
    The diagonal formulas show that $\mathbf{G}$ satisfies
    \eqref{eq:transferFunctionDifferenceRelation}.

    For any zero-state classical solution with input $u \in \mathrm C_c^\infty((0,\infty); X)$, the corresponding output is given coordinate-wise by  
    \[
        y_u(t)
        =-A_{-1}\int_0^t\bbT(t-s)u(s)\dx s
        =-\int_0^t\bbT_{-1}(t-s)A_{-1}u(s)\dx s.
    \]
    Since $A$ is unbounded,
    \cite[Theorem~2.9]{JacobSchwenningerWintermayr2022}
    implies that $A_{-1}$ is not $\lpSpace{\infty}$-admissible.
    Thus there exist $t>0$ and $u\in \lpSpaceDT{\infty}{[0,t]}{X}$
    such that
    $Q_t u\coloneqq
    \int_0^t\bbT_{-1}(t-s)A_{-1}u(s)\dx s\notin X$.
    Choose $u_n\in\mathrm C_c^\infty((0,t);X)$, extended
    by zero, such that $u_n\to u$ in $\lpSpaceDT{1}{[0,t]}{X}$ and
    $\norm{u_n}_{\lpSpace{\infty}}\le\norm{u}_{\lpSpace{\infty}}$.
    
    If the system were $\mathrm C^\infty$-BIBO stable, the
    output formula would imply that $(Q_t u_n)$ is bounded
    in $X$.
    Since $X$ is reflexive, a subsequence would converge
    weakly in $X$ to some $q\in X$.
    However, continuity of
    $Q_t:\lpSpaceDT{1}{[0,t]}{X}\to X_{-1}$
    gives $Q_t u_n\to Q_t u$ in $X_{-1}$.
    The continuous embedding $X\hookrightarrow X_{-1}$
    then implies $Q_t u=q\in X$, a contradiction.
    Hence the system is not $\mathrm C^\infty$-BIBO stable.
    
    Let $b,c\in\ell^2$ and write
    $c^*x=\sum_{n=1}^\infty\overline{c_n}x_n$.
    For the derived SISO system
    $\systemNode{A}{Bb}{c^*C}{c^*\mathbf{G} b}$,
    the transfer function is
    $c^*\mathbf{G}(z)b
    =\sum_{n=1}^\infty
        \overline{c_n}b_n\lambda_n/(z+\lambda_n)$.
    Since
    $h_{b,c}(\argument)\coloneqq
    \sum_{n=1}^\infty
        \overline{c_n}b_n\lambda_ne^{-\lambda_n \argument}$
    satisfies
    \[
        \norm{h_{b,c}}_{\lpSpaceD{1}{\bbR_+}}
        \le
        \sum_{n=1}^\infty
            \modulus{b_n}\modulus{c_n}
            \int_0^\infty\lambda_ne^{-\lambda_nt}\dx t
        =
        \sum_{n=1}^\infty\modulus{b_n}\modulus{c_n}
        \le\norm{b}_{\ell^2}\norm{c}_{\ell^2},
    \]
    the series converges absolutely in $\lpSpaceD{1}{\bbR_+}$.
    Termwise Laplace transformation therefore gives
    $\calL\{h_{b,c}\}(z)=c^*\mathbf{G}(z)b$ for $\re z>0$.
    By Theorem~\ref{thm:finiteDimBIBOCharacterisation},
    the derived system is BIBO stable with constant
    $\norm{b}_{\ell^2}\norm{c}_{\ell^2}$.
    In particular, $1$ is a common BIBO constant for all
    normalised derived SISO systems.
\end{example}

\section{Conditions for BIBO stability in terms of derived single-output systems}
\label{sec:BIBOfromScalarOutput}

In the previous section we saw that even uniform BIBO
stability of the derived SISO systems does not in general
imply BIBO stability of the full system.
Scalarising the input and the output, however, play
different roles.
Scalarising the output preserves the full class of
$U$-valued inputs, and the family of all output
scalarisations retains the norm information about each
classical output, since the functionals in $Y'$ determine
the norm on $Y$.
By contrast, scalarising the input fixes a direction
$u_0\in U$ and restricts the inputs to functions of the
form $u(t)=v(t)u_0$, where $v$ is scalar-valued.
Such inputs cannot change their direction in $U$ over time.
This suggests that the obstruction to recovering
$\mathrm C^\infty$-BIBO stability in
Section~\ref{sec:BIBOSISO} arises from the input
scalarisation.

Indeed, we show in the following that $\mathrm C^\infty$-BIBO stability of all derived scalar-output systems implies $\mathrm C^\infty$-BIBO stability of the full system.  We then derive kernel criteria
for this property and give additional conditions ensuring
$\lpSpace{\infty}$-BIBO stability.

\begin{theorem}
\label{thm:OutputScalarisationGivesBIBO}
    Assume that, for every $y'\in Y'$, the scalar-output
    system node $\systemNode{A}{B}{y'C}{y'\mathbf{G}}$ is
    $\mathrm C^\infty$-BIBO stable.
    Then $\systemNodeSt$ is $\mathrm C^\infty$-BIBO stable.
    
    Moreover, if $c(y')\ge 0$ denotes the optimal BIBO constant of each derived system, then the optimal BIBO constant of the full system is
    $K\coloneqq\sup_{\norm{y'}_{Y'}\le1}c(y')<\infty$.
\end{theorem}

\begin{proof}
    For $u\in\mathrm C_c^\infty((0,\infty);U)$, let $y_u$
    denote the corresponding zero-state classical output.
    The $\mathrm C^\infty$-BIBO estimate for the
    scalar-output system, together with continuity of its
    output, gives
    $\sup_{\tau\ge0}\modulus{\duality{y'}{y_u(\tau)}}
    \le c(y')\norm{u}_{\lpSpaceDT{\infty}{\bbR_+}{U}}$
    for every $y'\in Y'$.
    Consider the family $\calF\subseteq Y''$ consisting of
    the functionals
    $F_{u,\tau}(y')\coloneqq\duality{y'}{y_u(\tau)}$,
    where $\tau\ge0$ and
    $\norm{u}_{\lpSpaceDT{\infty}{\bbR_+}{U}}\le1$.
    The preceding estimate shows that
    $\sup_{F\in\calF}\modulus{F(y')}\le c(y')<\infty$
    for every $y'\in Y'$, and thus,
    by the uniform boundedness principle,
    $K_0\coloneqq\sup_{F\in\calF}\norm{F}_{Y''}<\infty$.

    Since $\norm{F_{u,\tau}}_{Y''}=\norm{y_u(\tau)}_Y$,
    homogeneity yields the global estimate
    \[
        \norm{y_u}_{\lpSpaceDT{\infty}{\bbR_+}{Y}}
        \le K_0\norm{u}_{\lpSpaceDT{\infty}{\bbR_+}{U}}
        \qquad
        \text{for every }
        u\in\mathrm C_c^\infty((0,\infty);U).
    \]
    Remark~\ref{rem:BIBO-constants-estimates}(a),
    applied to the full system, converts this into
    \[
        \norm{y_u}_{\lpSpaceDT{\infty}{[0,t]}{Y}}
        \le K_0\norm{u}_{\lpSpaceDT{\infty}{[0,t]}{U}}
        \qquad
        \text{for every }t>0.
    \]
    Thus $\systemNodeSt$ is $\mathrm C^\infty$-BIBO stable.

    It remains to identify its optimal constant.
    Every admissible BIBO constant $c$ for the full system
    gives $c(y')\le c\norm{y'}_{Y'}$, and hence
    $K\coloneqq\sup_{\norm{y'}_{Y'}\le1}c(y')\le c$.
    Taking $c=K_0$ shows that $K\le K_0<\infty$.
    Conversely, the scalar-output estimate and duality
    give $\norm{F}_{Y''}\le K$ for every $F\in\calF$,
    so $K_0\le K$.
    Consequently, $K=K_0$ is the optimal BIBO constant.
\end{proof}

\begin{remark}
\label{rem:output-scalarisation-only-gives-cinfty}
    Even $\lpSpace{\infty}$-BIBO stability of every derived
    scalar-output system does not imply $\lpSpace{\infty}$-BIBO
    stability of the full system.
    Indeed, the system in
    Example~\ref{ex:c-infty-but-not-l-infty} is
    $\mathrm C^\infty$-BIBO stable but not
    $\lpSpace{\infty}$-BIBO stable.
    Every derived scalar-output system is nevertheless
    $\lpSpace{\infty}$-BIBO stable by
    Theorem~\ref{thm:reflexive-output-cinfty-implies-linfty-bibo},
    since its output space is $\bbC$.

    If $Y$ is reflexive, the conclusion of
    Theorem~\ref{thm:OutputScalarisationGivesBIBO} does imply
    $\lpSpace{\infty}$-BIBO stability with the same optimal constant,
    again by
    Theorem~\ref{thm:reflexive-output-cinfty-implies-linfty-bibo}.
\end{remark}

\subsection{Kernel estimates and output representations}

We now derive kernel criteria for BIBO stability.
Two auxiliary lemmas provide the necessary uniform estimates
and identify generalised outputs from the transfer function.
We then apply these results to semigroup kernels and their
output scalarisations.

In what follows, for each operator valued function $k:\bbR_+\to\calL(U,Y)$, we define the \emph{scalarisation operator associated with $k$} acting on elements $y'\in Y'$ as 
\[
    \calS_ky'\coloneqq k(\argument)'y'.
\]
The following lemma describes a situation in which $\calS_k$ can be understood as a bounded operator from $Y'$ into $\lpSpaceDT{1}{\bbR_+}{U'}$.

\begin{lemma}
    \label{lem:weak-kernel-bounded-scalarization}
    Let $k:\bbR_+\to\calL(U,Y)$ be such that, for every $y'\in Y'$, 
    $
        k(\argument)'y'\in \lpSpaceDT{1}{\bbR_+}{U'}
    $.
     Then the associated scalarisation operator $\calS_k: Y'\to \lpSpaceDT{1}{\bbR_+}{U'}$
    is bounded with
    \[
        \norm{\calS_k}_{\calL(Y',\lpSpaceDT{1}{\bbR_+}{U'})}
        =
        \sup_{\norm{y'}_{Y'}\le1}
        \int_0^\infty
            \norm{k(t)'y'}_{U'}
        \dx t
        <\infty.
    \]
\end{lemma}

\begin{proof}
    Suppose that
    $
        y_n'\to y'
    $
    in $Y'$ and
    $
        \calS_k y_n'\to g
    $
    in $\lpSpaceDT{1}{\bbR_+}{U'}$. Passing to a subsequence, we may assume that
    $
        \calS_k y_n'(t)\to g(t)
    $
    in $U'$ for almost every $t>0$. On the other hand, $k(t)'y_n'\to k(t)'y'$ in $U'$ for every $t>0$ for which $k(t)\in\calL(U,Y)$.
    Hence
    $
        g=k(\argument)'y'
    $
    almost everywhere. Thus $\calS_k$ has closed graph, and the closed graph
    theorem yields the assertion.
\end{proof}

The next lemma identifies the generalised output when a
locally integrable kernel is compatible with the transfer
function. It supplies the output representation used in
Proposition~\ref{prop:kernel-bibo-corrected-SWZ} and again
in Section~\ref{subsec:kernel-linfty-bibo}.

\begin{lemma}
    \label{lem:weak-laplace-output-representation}
    Let $\systemNodeSt$ be a system node and let
    $k\in \lpSpace{1}_{\loc}(\bbR_+; \calL(U,Y))$.
    Assume that $k(\argument)'y'\in \lpSpaceDT{1}{\bbR_+}{U'}$
    for every $y'\in Y'$.
    Suppose that there exists $D\in\calL(U,Y)$ such that
    \[
        \duality{y'}{\mathbf{G}(\lambda)v}
        =
        \duality{y'}{D v}
        +
        \int_0^\infty e^{-\lambda t}
            \duality{k(t)'y'}{v}\dx t
    \]
    for every $v\in U$, every $y'\in Y'$, and every
    $\re\lambda>\max\{\gbd(\bbT),0\}$.
    Then every zero-state generalised output corresponding
    to $u\in \lpSpace{1}_{\loc}(\bbR_+; U)$ is represented by
    $Du+k*u\in \lpSpace{1}_{\loc}(\bbR_+; Y)$.
    Moreover, $\systemNodeSt$ is $\lpSpace{\infty}$-BIBO stable with
    constant
    $\norm{D}_{\calL(U,Y)}
    +\norm{\calS_k}_{\calL(Y',\lpSpaceDT{1}{\bbR_+}{U'})}$.
\end{lemma}

\begin{proof}
    For every $y'\in Y'$, set
    $H_{y'}(\lambda)\coloneqq
    \int_0^\infty e^{-\lambda \tau}k(\tau)'y'\dx \tau$.
    By assumption,
    \begin{equation}
    \label{eq:weak-laplace-compatibility-all-lambda}
        y'\mathbf{G}(\lambda)
        =
        y'D+H_{y'}(\lambda)
    \end{equation}
    for every $y'\in Y'$ and every $\lambda\in\bbC$ with
    $\re\lambda>\max\{\gbd(\bbT),0\}$.

    We first show that the conclusion is true for compactly supported smooth functions. To this end, let $v\in \mathrm C_c^\infty((0,\infty);U)$ and let $y_v$
    be the corresponding zero-state classical output, which
    exists by~\cite[Lemma~4.7.8]{Staffans2005}.
    By~\cite[Theorem~4.7.11]{Staffans2005}, this is also the
    generalised output.
    Define
    \[
        z_v(\tau)
        \coloneqq
        Dv(\tau)
        +
        \int_0^\tau k(\tau-s)v(s)\dx s.
    \]
    Fix $y'\in Y'$. Since $y'k\in \lpSpaceDT{1}{\bbR_+}{U'}$,
    the convolution theorem,
    \eqref{eq:weak-laplace-compatibility-all-lambda}, and
    \cite[Lemma~4.7.12]{Staffans2005} give
    \begin{align*}
        \calL\{\duality{y'}{z_v}\}(\lambda)
        &=
        \left(
            y'D+H_{y'}(\lambda)
        \right)
        \widehat v(\lambda)
        =
        y'\mathbf{G}(\lambda)\widehat v(\lambda)
        =
        \calL\{\duality{y'}{y_v}\}(\lambda)
    \end{align*}
    for all $\lambda\in\bbC$ with
    $\re\lambda>\max\{\gbd(\bbT),0\}$.
    By uniqueness of the Laplace transform,
    $\duality{y'}{y_v}=\duality{y'}{z_v}$
    for every $y'\in Y'$, and hence
    \[
        y_v(\tau)
        =
        Dv(\tau)
        +
        \int_0^\tau k(\tau-s)v(s)\dx s
        \quad
        \text{ as a $Y$-valued distribution.}
    \]

    Now let $u\in \lpSpace{1}_{\loc}(\bbR_+; U)$ and fix $t>0$.
    Choose $v_n\in \mathrm C_c^\infty((0,t);U)$ such that
    $v_n\to u$ in $\lpSpaceDT{1}{[0,t]}{U}$.
    By causality and
    Lemma~\ref{lem:L1-to-distributional-output-continuity},
    the corresponding outputs satisfy
    $y_{v_n}\to y_u$ distributionally on $(0,t)$.
    On the other hand, $Dv_n\to Du$ in $\lpSpaceDT{1}{[0,t]}{Y}$
    and Young's inequality gives
    \[
        \norm{
            k*(v_n-u)
        }_{\lpSpaceDT{1}{[0,t]}{Y}}
        \le
        \norm{k}_{\lpSpaceDT{1}{[0,t]}{\calL(U,Y)}}
        \norm{v_n-u}_{\lpSpaceDT{1}{[0,t]}{U}}
        \longrightarrow0.
    \]
    Consequently,
    \[
        Dv_n+k*v_n
        \longrightarrow
        Du+k*u
        \qquad
        \text{in }\lpSpaceDT{1}{[0,t]}{Y},
    \]
    and hence distributionally.
    As the assertion holds for compactly supported smooth functions,
    uniqueness of the distributional limit yields
    \[
        y_u(\tau)
        =
        Du(\tau)
        +
        \int_0^\tau k(\tau-s)u(s)\dx s
        \quad
        \text{for almost every }\tau\in(0,t).
    \]
    Since $t>0$ was arbitrary, the representation holds
    locally on $\bbR_+$.

    Now let $u\in \lpSpace{\infty}_{\loc}(\bbR_+; U)$ and fix $t>0$.
    By Lemma~\ref{lem:weak-kernel-bounded-scalarization},
    $\norm{\calS_k}_{\calL(Y', \lpSpaceDT{1}{\bbR_+}{U'})}<\infty$.
    The output representation and duality give, for almost every
    $\tau\in(0,t)$,
    \begin{align*}
        \norm{y_u(\tau)}_Y
        &\le
        \norm{D}_{\calL(U,Y)}
        \norm{u}_{\lpSpaceDT{\infty}{[0,t]}{U}}
        +\sup_{\norm{y'}_{Y'}\le1}
          \int_0^\tau
            \norm{k(\tau-s)'y'}_{U'}\norm{u(s)}_U\dx s\\
        &\le
        \bigl(\norm{D}_{\calL(U,Y)}+\norm{\calS_k}_{\calL(Y', \lpSpaceDT{1}{\bbR_+}{U'})}\bigr)
        \norm{u}_{\lpSpaceDT{\infty}{[0,t]}{U}}.
    \end{align*}
    Taking the essential supremum over $\tau\in(0,t)$ proves
    the asserted BIBO estimate.
\end{proof}

For the kernel criteria below we first record a corrected
form of
\cite[Proposition~4.3]{SchwenningerWierzbaZwart2024}.
In its proof, a resolvent difference is represented by a
semigroup integral converging in $X_{-1}$, and the
observation operator is then passed through this integral.
This passage requires further justification: $C_Z$ is
continuous on $Z$, whereas the stated assumptions do not
ensure convergence of the integral in $Z$.
In particular, integrability of
$C_Z\bbT_{-1}(\argument)B$ does not imply integrability
of $\bbT_{-1}(\argument)B$ in the $Z$-norm.

The additional assumptions below ensure that the semigroup
restricts to a $C_0$-semigroup on $Z$.
Applying its Laplace-resolvent formula to suitable shifted
control operators then justifies the required passage.
This yields
$\mathbf{G}=D+\calL\{C_Z\bbT_{-1}(\argument)B\}$
on a common right half-plane, with $D\in\calL(U,Y)$.

\begin{proposition}
    \label{prop:kernel-bibo-corrected-SWZ}
    Let $\systemNodeSt$ be a system node and let $Z$ be a Banach space
    satisfying the following conditions:
    \begin{enumerate}[\upshape (a)]
        \item \label{prop:kernel-bibo-corrected-SWZ:itm:embeddings}
        The embeddings $X_1 \hookrightarrow Z \hookrightarrow X_{-1}$ are continuous,
        
        \item \label{prop:kernel-bibo-corrected-SWZ:itm:observation-extends}
        $C$ has an extension $C_Z\in\calL(Z,Y)$,
        
        \item \label{prop:kernel-bibo-corrected-SWZ:itm:input-in-Z}
        $\bbT_{-1}(t)B\in\calL(U,Z)$ for almost every $t>0$, and
        
        \item \label{prop:kernel-bibo-corrected-SWZ:itm:Z-invariant}
        $\bbT_{-1}(t)Z\subseteq Z$ for every $t\ge0$ and the
        restrictions $\bbT_Z(t)\coloneqq\bbT_{-1}(t)\restrict{Z}$
        form a $C_0$-semigroup on $Z$.
    \end{enumerate}
    If $C_Z\bbT_{-1}(\argument)B\in \lpSpaceDT{1}{\bbR_+}{\calL(U,Y)}$, then $\systemNodeSt$ is $\lpSpace{\infty}$-BIBO stable.
\end{proposition}

\begin{proof}
    Set $k\coloneqq C_Z\bbT_{-1}(\argument)B$ and
    $\mathbf{H}(s)\coloneqq
    \int_0^\infty e^{-st}k(t)\dx{t}$ for $\re s>0$.
    Put $\alpha\coloneqq
    \max\{\gbd(\bbT),\gbd(\bbT_Z),0\}$.

    \emph{Step 1}:
    We modify the resolvent argument in
    \cite[Proposition~4.3]{SchwenningerWierzbaZwart2024}
    by first shifting the control operator.
    By assumption
    \ref{prop:kernel-bibo-corrected-SWZ:itm:input-in-Z},
    there exists a sequence $t_n>0$ with $t_n\to0$ such that
    $B_n\coloneqq\bbT_{-1}(t_n)B\in\calL(U,Z)$.
    Then
    $C_Z\bbT_Z(t)B_n
    =C_Z\bbT_{-1}(t+t_n)B=k(t+t_n)$,
    and hence
    $C_Z\bbT_Z(\argument)B_n
    \in \lpSpaceDT{1}{\bbR_+}{\calL(U,Y)}$.
    Thus
    $\mathbf{H}_n(s)\coloneqq
    \int_0^\infty e^{-st}C_Z\bbT_Z(t)B_n\dx{t}$
    exists as an integral in $\calL(U,Y)$ for $\re s>0$.

    Furthermore, since $(\bbT_Z(t))_{t\ge0}$ is a
    $C_0$-semigroup and $Z$ embeds continuously into $X_{-1}$,
    the Laplace representation for the resolvent gives, for
    $s\in\bbC_\alpha$ and $z\in Z$,
    \begin{equation*}
        \resolvent{s}{A_{-1}}z
        =
        \int_0^\infty e^{-st}\bbT_Z(t)z\dx{t}
    \end{equation*}
    as an integral in $Z$.
    By continuity of $C_Z$, we therefore obtain
    $\mathbf{H}_n(s)u
    =C_Z\resolvent{s}{A_{-1}}B_nu$.

    Since the product of two extrapolated resolvents maps
    $X_{-1}$ continuously into $X_1$, the resolvent identity
    and the fact that $C_Z$ extends $C$ give
    \begin{align*}
        (\mathbf{H}_n(s)-\mathbf{H}_n(r))u
        &=
        C_Z\left(
            \resolvent{s}{A_{-1}}
            -\resolvent{r}{A_{-1}}
        \right)B_nu
        =
        (r-s)C
        \resolvent{s}{A_{-1}}
        \resolvent{r}{A_{-1}}B_nu
    \end{align*}
    for every $u\in U$ and $r,s\in\bbC_\alpha$.

    A change of variables gives
    $\mathbf{H}_n(s)
    =e^{st_n}\bigl(
        \mathbf{H}(s)-\int_0^{t_n}e^{-st}k(t)\dx{t}
    \bigr)$.
    Thus $\mathbf{H}_n(s)\to\mathbf{H}(s)$ in $\calL(U,Y)$
    for every $\re s>0$.
    Moreover, $B_nu\to Bu$ in $X_{-1}$ by strong continuity
    of $\bbT_{-1}$.
    Since the operator
    $C\resolvent{s}{A_{-1}}\resolvent{r}{A_{-1}}$
    is bounded from $X_{-1}$ to $Y$, passing to the limit yields
    \begin{equation*}
        (\mathbf{H}(s)-\mathbf{H}(r))u
        =
        (r-s)C
        \resolvent{s}{A_{-1}}
        \resolvent{r}{A_{-1}}Bu
        \quad
        \text{for every }u\in U
        \text{ and }r,s\in\bbC_\alpha.
    \end{equation*}
    By the system-node transfer-function identity, this equals
    $(\mathbf{G}(s)-\mathbf{G}(r))u$.
    Consequently, $\mathbf{G}=D+\mathbf{H}$ on $\bbC_\alpha$
    for some $D\in\calL(U,Y)$.
    Analytic continuation extends this identity to
    $\bbC_{\max\{\gbd(\bbT),0\}}$.

    \emph{Step 2}:
    Since $k\in \lpSpaceDT{1}{\bbR_+}{\calL(U,Y)}$, all its scalarisations
    belong to $\lpSpaceDT{1}{\bbR_+}{U'}$.
    Step~1 supplies the compatibility identity required in
    Lemma~\ref{lem:weak-laplace-output-representation}.
    Hence every zero-state generalised output corresponding
    to $u\in \lpSpace{1}_{\loc}(\bbR_+; U)$ is represented by $Du+k*u$.
    For $u\in \lpSpace{\infty}_{\loc}(\bbR_+; U)$, Young's inequality
    gives
    $\norm{y_u}_{\lpSpaceDT{\infty}{[0,t]}{Y}}
    \le(\norm{D}_{\calL(U,Y)}+\norm{k}_{\lpSpace{1}})
    \norm{u}_{\lpSpaceDT{\infty}{[0,t]}{U}}$
    for every $t>0$.
    This proves $\lpSpace{\infty}$-BIBO stability.

    Finally, for every real
    $\lambda_0>\max\{\gbd(\bbT),0\}$, the identity
    $D=\mathbf{G}(\lambda_0)-\mathbf{H}(\lambda_0)$ implies
    $\norm{D}_{\calL(U,Y)}+\norm{k}_{\lpSpace{1}}
    \le
    \norm{\mathbf{G}(\lambda_0)}_{\calL(U,Y)}
    +2\norm{k}_{\lpSpace{1}}$.
\end{proof}

\begin{remarks}
    \label{rem:kernel-bibo-corrected-SWZ}
    (a) Let $\systemNodeSt$ be a system node such that $\left( \bbT(t) \right)_{t\ge0}$ is analytic and $C\bbT_{-1}(\argument)B\in \lpSpaceDT{1}{\bbR_+}{\calL(U,Y)}$. Then conditions~\ref{prop:kernel-bibo-corrected-SWZ:itm:embeddings}--\ref{prop:kernel-bibo-corrected-SWZ:itm:Z-invariant} in Proposition~\ref{prop:kernel-bibo-corrected-SWZ} are fulfilled with $Z= X_1$. Consequently, $\systemNodeSt$ is $\lpSpace{\infty}$-BIBO stable.
    
    (b) Under the hypotheses of
    Proposition~\ref{prop:kernel-bibo-corrected-SWZ}, set
    $k\coloneqq C_Z\bbT_{-1}(\argument)B$.
    Its proof shows that there exists $D\in\calL(U,Y)$ such that
    $\mathbf{G}(\lambda)=D+\int_0^\infty e^{-\lambda t}k(t)\dx t$
    for $\re\lambda>\max\{\gbd(\bbT),0\}$, and every
    zero-state generalised output corresponding to
    $u\in \lpSpace{\infty}_{\loc}(\bbR_+; U)$ is represented by
    \[
        y(t)=Du(t)+\int_0^t k(t-s)u(s)\dx s
        \quad \text{for almost every }t>0.
    \]
    The BIBO constant is at most $\norm{D}_{\calL(U,Y)}+\norm{k}_{\lpSpaceDT{1}{\bbR_+}{\calL(U,Y)}} \le 2\norm{k}_{\lpSpace{1}}+\norm{\mathbf{G}(\lambda_0)}_{\calL(U,Y)}$ for every real $\lambda_0>\max\{\gbd(\bbT),0\}$. 
\end{remarks}

We now replace integrability of the kernel in operator norm
by integrability after output scalarisation.
This weaker condition still yields $\mathrm C^\infty$-BIBO
stability by Theorem~\ref{thm:OutputScalarisationGivesBIBO}.

\begin{corollary}
    \label{cor:weak-dual-kernel-cinfty-bibo}
    Let $\systemNodeSt$ be a system node and let $Z$ be a Banach space
    satisfying conditions~\ref{prop:kernel-bibo-corrected-SWZ:itm:embeddings}--\ref{prop:kernel-bibo-corrected-SWZ:itm:Z-invariant} in Proposition~\ref{prop:kernel-bibo-corrected-SWZ}.
    If $\bigl(C_Z\bbT_{-1}(\argument)B\bigr)'y'\in \lpSpaceDT{1}{\bbR_+}{U'}$ for each $y'\in Y'$, then $\systemNodeSt$ is $\mathrm C^\infty$-BIBO stable.
\end{corollary}

\begin{proof}
    Set $k\coloneqq C_Z\bbT_{-1}(\argument)B$, extending
    it by zero on the exceptional null set.
    For every $y'\in Y'$, the scalarised kernel
    $y'k=k(\argument)'y'$ belongs to
    $\lpSpaceDT{1}{\bbR_+}{\calL(U,\bbC)}$.
    Proposition~\ref{prop:kernel-bibo-corrected-SWZ},
    applied with observation extension $y'C_Z$,
    therefore shows that
    $\systemNode{A}{B}{y'C}{y'\mathbf{G}}$ is
    $\lpSpace{\infty}$-BIBO stable.
    Theorem~\ref{thm:OutputScalarisationGivesBIBO}
    yields $\mathrm C^\infty$-BIBO stability of
    $\systemNodeSt$.
\end{proof}

\begin{remark}
\label{rem:weak-kernel-bibo-bound}
    Under the hypotheses of
    Corollary~\ref{cor:weak-dual-kernel-cinfty-bibo}, set
    $k\coloneqq C_Z\bbT_{-1}(\argument)B$.
    For every real $\lambda_0>\max\{\gbd(\bbT),0\}$,
    an admissible $\mathrm C^\infty$-BIBO constant is
    $2\norm{\calS_k}
    +\norm{\mathbf{G}(\lambda_0)}_{\calL(U,Y)}$.

    Indeed, for $y'\in Y'$ and $\re\lambda>0$, define
    $H_{y'}(\lambda)\coloneqq
    \int_0^\infty e^{-\lambda t}k(t)'y'\dx t$.
    Remark~\ref{rem:kernel-bibo-corrected-SWZ}(b),
    applied to the scalar-output system, gives
    $H_{y'}(\lambda)=y'\mathbf{G}(\lambda)+M_{y'}$
    for some $M_{y'}\in U'$ whenever
    $\re\lambda>\max\{\gbd(\bbT),0\}$.
    Evaluating at $\lambda_0$ yields
    $\norm{M_{y'}}_{U'}
    \le\norm{y'\mathbf{G}(\lambda_0)}_{U'}
    +\norm{\calS_k y'}_{\lpSpace{1}}$.
    Hence the optimal scalar-output BIBO constant satisfies
    \[
        c(y')
        \le
        \norm{M_{y'}}_{U'}+\norm{\calS_k y'}_{\lpSpace{1}}
        \le
        \bigl(
            \norm{\mathbf{G}(\lambda_0)}_{\calL(U,Y)}
            +2\norm{\calS_k}
        \bigr)\norm{y'}_{Y'}.
    \]
    Theorem~\ref{thm:OutputScalarisationGivesBIBO}
    gives the asserted constant for the full system.
\end{remark}

\subsection{Extension to \texorpdfstring{$\lpSpace{\infty}$}{L-infinity}-BIBO stability}
\label{subsec:kernel-linfty-bibo}

By Corollary~\ref{cor:cinfty-bibo-plus-linfty-outputs-implies-linfty-bibo},
the $\mathrm C^\infty$-BIBO estimate extends to
$\lpSpace{\infty}$-BIBO stability once locally Bochner-integrable
output representatives exist.
We consider two sufficient conditions for their existence:
local integrability of the kernel in operator norm and
reflexivity of the output space.

In the setting of
Corollary~\ref{cor:weak-dual-kernel-cinfty-bibo},
local integrability of the kernel in operator norm ensures
the compatibility required in
Lemma~\ref{lem:weak-laplace-output-representation}.
This yields an explicit output representation and
$\lpSpace{\infty}$-BIBO stability.

\begin{corollary}
    \label{cor:weak-laplace-compatibility-linfty-bibo}
    Let $\systemNodeSt$ be a system node, let $Z$ be a Banach space
    satisfying conditions~\ref{prop:kernel-bibo-corrected-SWZ:itm:embeddings}--\ref{prop:kernel-bibo-corrected-SWZ:itm:Z-invariant} in Proposition~\ref{prop:kernel-bibo-corrected-SWZ}, and set
    $k\coloneqq C_Z\bbT_{-1}(\argument)B$.
    Assume further that $k(\argument)'y'\in \lpSpaceDT{1}{\bbR_+}{U'}$ for each $y'\in Y'$ and $k\in \lpSpace{1}_{\loc}(\bbR_+; \calL(U,Y))$.
    Then $\systemNodeSt$ is $\lpSpace{\infty}$-BIBO stable.
\end{corollary}

\begin{proof}
    \emph{Step 1}:
    Set $M\coloneqq\norm{\calS_k}$, which is finite by Lemma~\ref{lem:weak-kernel-bounded-scalarization}.
    For $\re\lambda=a>0$ and $t>0$, define
    $H_t(\lambda)\coloneqq\int_0^t e^{-\lambda s}k(s)\dx s$.
    By local integrability, this is a Bochner integral in
    $\calL(U,Y)$. For $R>t$, duality gives
    \[
        \norm{H_R(\lambda)-H_t(\lambda)}_{\calL(U,Y)}
        \le
        \sup_{\norm{y'}_{Y'}\le1}
        \int_t^R e^{-as}\norm{k(s)'y'}_{U'}\dx s
        \le e^{-at}M.
    \]
    Hence $H(\lambda)\coloneqq\lim_{t\to\infty}H_t(\lambda)$
    exists in operator norm, and $\norm{H(\lambda)}\le M$.
    This defines an improper integral in operator norm.
    Moreover, $y'H(\lambda)=H_{y'}(\lambda)$ for every $y'\in Y'$,
    with $H_{y'}$ as in Remark~\ref{rem:weak-kernel-bibo-bound}.

    \emph{Step 2}:
    Fix a real $\lambda_0>\max\{\gbd(\bbT),0\}$ and set
    $D\coloneqq\mathbf{G}(\lambda_0)-H(\lambda_0)$.
    Remark~\ref{rem:weak-kernel-bibo-bound} gives
    $M_{y'}=H_{y'}(\lambda_0)-y'\mathbf{G}(\lambda_0)=-y'D$.
    Thus $y'\mathbf{G}(\lambda)=y'D+H_{y'}(\lambda)$ for every
    $y'\in Y'$ whenever $\re\lambda>\max\{\gbd(\bbT),0\}$.
    Consequently,
    \[
        \mathbf{G}(\lambda)=D+H(\lambda)
        \quad \text{for }
        \lambda \in \bbC_{\gbd(\bbT)} \cap \bbC_0.
    \]
    This also shows that $D$ is independent of the choice of
    $\lambda_0$ and that $\norm{D}_{\calL(U,Y)}
    \le\norm{\mathbf{G}(\lambda_0)}_{\calL(U,Y)}+M$.
    Lemma~\ref{lem:weak-laplace-output-representation} then
    gives the output representation and the BIBO constant
    $\norm{D}_{\calL(U,Y)}+M$.
\end{proof}

\begin{remark}
    The proof of
    Corollary~\ref{cor:weak-laplace-compatibility-linfty-bibo}
    also gives $D\in\calL(U,Y)$ such that every zero-state
    generalised output corresponding to
    $u\in \lpSpace{1}_{\loc}(\bbR_+; U)$ is represented by $Du+k*u$.
    An admissible BIBO constant is
    $\norm{D}_{\calL(U,Y)}+\norm{\calS_k}$, and
    \[
        \norm{D}_{\calL(U,Y)}+\norm{\calS_k}
        \le
        2\norm{\calS_k}
        +\norm{\mathbf{G}(\lambda_0)}_{\calL(U,Y)}
    \]
    for every real $\lambda_0>\max\{\gbd(\bbT),0\}$.
\end{remark}

\begin{example}
    Note that Example~\ref{ex:systemNodeCounterexampleBddVariation} is covered by
    Corollary~\ref{cor:weak-laplace-compatibility-linfty-bibo}. Indeed, in that
    example one has
    \[
        k(t)u
        =
        C\bbT(t)Bu
        =
        \left(\xi\mapsto \xi e^{-t\xi}u(\xi)\right),
        \qquad
        t>0.
    \]
    Take $Z=X$ and $C_Z=C$.
    Since $A$, $B$, and $C$ are bounded, the structural
    assumptions of
    Proposition~\ref{prop:kernel-bibo-corrected-SWZ}
    hold, and $k$ is norm-continuous.
    In particular,
    $k\in \lpSpace{1}_{\loc}(\bbR_+; \calL(U,Y))$.
    Identify $y'\in Y'$ with a regular Borel measure $\mu$
    on $[0,1]$.
    Then $k(t)'y'$ has density $\xi e^{-t\xi}$ with respect
    to $\mu$, and depends continuously on $t$ in $U'$.
    Tonelli's theorem gives
    \[
        \int_0^\infty\norm{k(t)'y'}_{U'}\dx t
        =
        \int_{[0,1]}\int_0^\infty
            \xi e^{-t\xi}\dx t\dx\modulus{\mu}(\xi)
        =
        \modulus{\mu}((0,1])
        \le\norm{y'}_{Y'}.
    \]
    Taking $\mu=\delta_1$ shows that $\norm{\calS_k}=1$.
    Moreover, for $\re s>0$,
    \[
        \mathbf{G}(s)u
        =
        \resolvent{s}{A}Bu
        =
        \left(
            \xi\mapsto
            \frac{\xi}{s+\xi}u(\xi)
        \right).
    \]
    Since $\norm{k(t)}\le1$, the integral
    $\int_0^\infty e^{-st}k(t)\dx t$ exists in operator norm
    for $\re s>0$.
    Integrating the multiplier gives
    $\int_0^\infty e^{-st}k(t)\dx t=\mathbf{G}(s)$.
    Thus $D=0$, and
    Corollary~\ref{cor:weak-laplace-compatibility-linfty-bibo}
    gives $\lpSpace{\infty}$-BIBO stability with constant
    $\norm{\calS_k}=1$.
\end{example}

Combining Corollary~\ref{cor:weak-dual-kernel-cinfty-bibo}
with Theorem~\ref{thm:reflexive-output-cinfty-implies-linfty-bibo}
yields the following criterion for reflexive output spaces.

\begin{corollary}
    \label{cor:weak-dual-kernel-linfty-bibo-reflexive}
    Let $\systemNodeSt$ be a system node, let $Z$ be a Banach space
    satisfying conditions~\ref{prop:kernel-bibo-corrected-SWZ:itm:embeddings}--\ref{prop:kernel-bibo-corrected-SWZ:itm:Z-invariant} in Proposition~\ref{prop:kernel-bibo-corrected-SWZ}, such that $\bigl(C_Z\bbT_{-1}(\argument)B\bigr)'y'\in \lpSpaceDT{1}{\bbR_+}{U'}$ for each $y'\in Y'$.
    If $Y$ is reflexive, then $\systemNodeSt$ is
    $\lpSpace{\infty}$-BIBO stable.
\end{corollary}

\begin{remark}
    \label{rem:reflexive-separability-weaker}
    Under the hypotheses of
    Corollary~\ref{cor:weak-dual-kernel-linfty-bibo-reflexive},
    set $k= C_Z\bbT_{-1}(\argument)B$ and
    $M\coloneqq\norm{\calS_k}$.
    The identity
    $M_{y'}=H_{y'}(\lambda_0)-y'\mathbf{G}(\lambda_0)$
    from Remark~\ref{rem:weak-kernel-bibo-bound}
    shows that $y'\mapsto M_{y'}$ is linear and bounded.
    By reflexivity of $Y$, there exists $D\in\calL(U,Y)$
    such that $\duality{y'}{Dv}=-M_{y'}v$
    for all $v\in U$ and $y'\in Y'$.
    Moreover,
    \[
    \label{eq:D0-bound}
        \norm{D}_{\calL(U,Y)}
        \le M+\norm{\mathbf{G}(\lambda_0)}_{\calL(U,Y)}
        \qquad
        \bigl(\lambda_0>\max\{\gbd(\bbT),0\},\
              \lambda_0\in\bbR\bigr).
    \]
    Applying Lemma~\ref{lem:weak-laplace-output-representation}
    to each scalar-output system, with kernel $y'k$ and
    feedthrough $y'D$, gives
    \[
        \duality{y'}{y_u(t)}
        =
        \duality{y'}{Du(t)}
        +\int_0^t\duality{k(r)'y'}{u(t-r)}\dx r
    \]
    for every $u\in \lpSpace{\infty}_{\loc}(\bbR_+; U)$,
    every $y'\in Y'$, and almost every $t>0$.
    This weak representation does not require local integrability
    of $k$ in operator norm.

    The scalar-output BIBO constants are at most
    $\norm{y'D}_{U'}
    +\norm{\calS_k y'}_{\lpSpaceDT{1}{\bbR_+}{U'}}
    \le(\norm{D}_{\calL(U,Y)}+M)\norm{y'}_{Y'}$.
    Theorems~\ref{thm:OutputScalarisationGivesBIBO} and~\ref{thm:reflexive-output-cinfty-implies-linfty-bibo}
    therefore give the BIBO constant
    $M+\norm{D}_{\calL(U,Y)}$ for the full system.
\end{remark}

\section*{Acknowledgements} 
    The authors thank Jochen Glück and Felix Schwenninger for various fruitful discussions.

\bibliographystyle{plainurl}
\bibliography{literature}

\end{document}